\documentclass[final,leqno]{siamltex704}
\usepackage{epsfig}
\usepackage{amsmath}
\usepackage{amssymb}
\usepackage{graphicx}
\usepackage[notcite,notref]{showkeys}

\newtheorem{algorithm}{Weak Galerkin Algorithm}
\newcommand{\bq}{\begin{equation}}
\newcommand{\eq}{\end{equation}}

\def\bn{{\bf n}}
\def\bq{{\bf q}}

\def\3bar{{|\hspace{-.02in}|\hspace{-.02in}|}}

\begin{document}

\title{Weak Galerkin methods for second order elliptic interface problems}
\author{lin Mu\thanks{Department of
Applied Science, University of Arkansas at Little Rock,
Little Rock, AR 72204 (lxmu@ualr.edu).} \and Junping
Wang\thanks{Division of Mathematical Sciences, National Science
Foundation, Arlington, VA 22230 (jwang@\break nsf.gov). The research
of Wang was supported by the NSF IR/D program, while working at the
Foundation. However, any opinion, finding, and conclusions or
recommendations expressed in this material are those of the author
and do not necessarily reflect the views of the National Science
Foundation.} \and Guowei Wei\thanks{Department of Mathematics, Michigan State University, East Lansing, MI 48824 (wei@math.msu.edu).
This research of Wei was  supported in part by
National Science Foundation Grant  CCF-0936830 and  NIH Grant R01GM-090208.}
\and Xiu Ye\thanks{Department of Mathematics and
Statistics, University of Arkansas at Little Rock, Little Rock, AR
72204 (xxye@ualr.edu). This research of Ye was supported in part by
National Science Foundation Grant DMS-1115097.} \and Shan
Zhao\thanks{Department of Mathematics, University of Alabama,
Tuscaloosa, AL 35487 (szhao@bama.ua.edu). The research of Zhao was
supported in part by National Science Foundation Grant DMS-1016579.
}}

\maketitle

\begin{abstract}
Weak Galerkin methods refer to general finite element methods for
partial differential equations (PDEs)
in which differential operators are approximated by their weak
forms as distributions. Such weak forms give rise to desirable
flexibilities in enforcing boundary and interface conditions. A weak
Galerkin finite element method (WG-FEM) is developed in this paper
for solving elliptic PDEs with
discontinuous coefficients and interfaces. The paper also presents
many numerical tests for validating the WG-FEM for solving second
order elliptic interface problems. For such interface problems, the
solution possesses a certain singularity due to the nonsmoothness of
the interface. A challenge in research is to design high order
numerical methods that work well for problems with low regularity in
the solution. The best known numerical scheme in the literature is
of order $\mathcal{O}(h)$ for the solution itself in $L_\infty$
norm. It is demonstrated that the WG-FEM of lowest order is capable
of delivering numerical approximations that are of order
$\mathcal{O}(h^{1.75})$ in the usual $L_\infty$ norm for $C^1$ or
Lipschitz continuous interfaces associated with a $C^1$ or $H^2$
continuous solution. Theoretically, it is proved that high order of
numerical schemes can be designed by using the WG-FEM with
polynomials of high order on each element.
\end{abstract}

\begin{keywords}
Finite element methods,  weak Galerkin method, second order elliptic
interface problems, nonsmooth interface, low solution regularity.
\end{keywords}

\begin{AMS}
Primary, 65N15, 65N30, 76D07; Secondary, 35B45, 35J50
\end{AMS}
\pagestyle{myheadings}

\section{Introduction}

In this paper, we are concerned with interface problems for second
order elliptic partial differential equations (PDEs) with
discontinuous coefficients and singular sources. For simplicity,
consider the model problem of seeking functions $u=u(x,y)$ and
$v=v(x,y)$ satisfying
\begin{eqnarray} 
-\nabla\cdot A_1\nabla u&=&f_1, \quad\mbox{  in }\Omega_1,\label{equ}\\
-\nabla\cdot A_2\nabla v&=&f_2, \quad\mbox{  in }\Omega_2,\label{eqv}\\
u&=&g_1, \quad\mbox{  on }\partial\Omega_1\setminus\Gamma,\label{bc}\\
v&=&g_2, \quad\mbox{  on }\partial\Omega_2\setminus\Gamma,\label{bc1}\\ 
u-v&=&\phi, \quad\mbox{  on }\Gamma,\label{ic1}\\  
A_1\nabla u\cdot\bn_1+A_2\nabla v\cdot\bn_2&=&\psi, \quad\mbox{  on
}\Gamma,\label{ic2}
\end{eqnarray}
where  $\Omega=\Omega_1\cup\Omega_2$, $\Gamma=\Omega_1\cap\Omega_2$,
$\partial\Omega_1\setminus\Gamma\neq\{\emptyset\}$ and $\bn_1$ and
$\bn_2$ are outward normals of $\Omega_1$ and $\Omega_2$. Here,
$f_1$ and $f_2$ can be singular and  $\Gamma$ may be of Lipschitz
continuous. This problem is commonly referred to as
an elliptic interface problem and  occurs widely in practical applications, such as
fluid mechanics 
\cite{Layton:2009}, electromagnetic wave propagation
\cite{Hadley:2002,Hesthaven:2003,Zhao:2004,SZhao:2010a},
materials science \cite{Horikis:2006,Hou:1997}, and biological
science \cite{Yu:2007,Geng:2007a,DuanChen:2011a}.
The finite difference based solution of elliptic interface problems was pioneered
by  Peskin with his immersed boundary method (IBM) in 1977 \cite{Peskin:1977,Peskin:1989}. 
Mayo constructed an interesting integral equation approach to this class of problems \cite{Mayo:1984}.
To properly solve  the elliptic interface problem, one needs to enforce additional interface conditions
(\ref{ic1}) and  (\ref{ic2}).
 LeVeque and Li advanced the subject  with their second order sharp
interface scheme, the immersed interface method (IIM)  \cite{LeVeque:1994}.
In the past  decades, many other elegant methods have been proposed, 
including  the ghost fluid method (GFM) proposed by Fedkiw, Osher and coworkers \cite{Fedkiw:1999}, 
finite-volume-based methods \cite{Oevermann:2006},
the piecewise-polynomial discretization
\cite{TChen:2008},
and matched interface and boundary (MIB) method \cite{Zhao:2004,Zhou:2006c,Yu:2007a}.

A proof of second order convergence of the IIM for smooth interfaces
was due  to  Beale and Layton \cite{Beale:2006}. Rigorous
convergence analysis of most other finite difference based elliptic
interface schemes is not available yet. In general, it is quite
difficult to analyze the convergence  of  finite difference based
interface schemes because conventional techniques used in Galerkin
formulations  are not applicable for collocation schemes. The
analysis becomes particularly difficult when the designed elliptic
interface scheme is capable of dealing with nonsmooth interfaces
\cite{Yu:2007a}. At present, there is no rigorous convergence
analysis available for elliptic interface methods that deliveries
high-order accuracy for nonsmooth interfaces, to the author's
knowledge.

Finite element methods (FEMs)   are another class of important
approaches for elliptic interface problems. The construction of  FEM
solutions to elliptic interface problems dates back to 1970s
\cite{Babuska:1970}, and has been a subject of intensive
investigation in the past few decades
\cite{Ewing:1999,Ramiere:2008,XHe:2010,ZCai:2011}. Since the
elliptic interface problem defined in Eqs. (\ref{equ})-(\ref{ic2})
provides opportunities to construct new FEM schemes, a wide variety
of FEM approaches have been proposed in the literature. There are
two major classes of FEM based interface methods, namely,
interface-fitted FEMs and immersed FEMs, categorized according to
the topological relation between discrete elements and the
interface. In the interface-fitted FEMs, or body-fitted FEMs, the
finite element mesh is designed to align with the interface. Local
mesh refinement based  priori and/or posteriori error estimation can
be easily carried out. The performance of interface-fitted FEMs
depends on the quality of the element mesh near the interface as
well as the formulation of the problem. In fact, the construction of
high quality FEM meshes for real world complex interface geometries
is an active area of research. Bramble and King discussed a FEM for
nonhomogeneous second order elliptic interface problems on smooth
domains \cite{Bramble:1996}. One way to deal with embedded interface
conditions is to use  distributed Lagrange multipliers
\cite{Burman:2010}, In fact,  similar ideas have been widely used in
mortar methods and fictitious domain methods  for the treatment of
embedded boundaries \cite{Glowinski:1994}. A $Q_1$-nonconforming
finite element method was also proposed for elliptic interface
problems \cite{Ramiere:2008}. Discontinuous Galerkin (DG) FEMs have
been developed for elliptic equations with discontinuous
coefficients \cite{Dryjaa:2007,XHe:2010}.  Inherited from the
original DG method, DG based interface schemes have the flexibility
to implement interface jump conditions.

Immersed FEMs are also effective approaches for embedded interface
problems
\cite{Ewing:1999,XWang:2009,Hansbo:2002,Hou:2010,Harari:2010}. A key
feature of these approaches is that their element meshes are
independent of the interface geometry, i.e., the interface usually
cuts through elements.
As such, there is no need to use the unstructured mesh to body-fit
the interface, and simple structured Cartesian meshes can be
employed in immersed FEMs. Consequently, the time-consuming meshing
process is bypassed in immersed FEMs. However, to deal with complex
interface geometries, it is necessary to design appropriate
interface algorithms, which is similar to finite-difference based
elliptic interface methods. In fact, immersed FEMs can be regarded
as the Galerkin formulations of finite difference based interface
schemes.  It is not surprised that key ideas of many immersed FEMs
actually come from the corresponding finite-difference based
interface schemes. Additionally, other ideas in numerical analysis
have been utilized for the construction of immersed FEMs. For
example, a stabilized Lagrange multiplier method based on Nitsche
technique has been used to enforce interface jump constraints
\cite{Harari:2010}. The performance of immersed FEMs depends on the
design of elegant interface schemes for complex interface
geometries.

Convergence analysis of FEM based elliptic interface methods has been
considered by many researchers. Unlike the collocation formulation,
the Galerkin framework of FEMs allows more rigorous and robust
convergence analysis. Cai {\it el al} gave a proof of convergence
for a DG FEM for interface problems \cite{ZCai:2011}. Dryja {\it et
al} discussed the convergence of the DG discretization of Dirichlet
problems for second-order elliptic equations with discontinuous
coefficients in 2D \cite{Dryjaa:2007}. Hiptmair {\it et al}
presented a convergence analysis of $H(div;\Omega)$-elliptic
interface problems in general 3D Lipschitz domains with smooth
material interfaces \cite{Hiptmair:2010}. Recently, an edge-based
anisotropic mesh refinement algorithm has been analyzed and applied
to elliptic interface problems \cite{DWang:2010}.

Despite of numerous advancements in the numerical solution of
interface problems, there are still a few remaining challenges in
the field. One of these challenges concerns  the construction of
higher order  interface schemes. Currently, most interface schemes
are designed to be of second order convergence. However, higher
order methods are efficient and desirable for problems associated
with high frequency waves, such as electromagnetic  and acoustic
wave propagation and scattering, vibration analysis of engineering
structures, and shock-vortex interaction in compressible fluid
flows. It is easy to construct high order methods and even spectral
methods for these problems with straight interfaces in simple
domains. However, it is extremely difficult to obtain high order
convergence when the interface geometries are arbitrarily complex. A
fourth order MIB scheme has been developed for the Helmholtz
equation in media with arbitrarily curved interfaces in 2D
\cite{SZhao:2010a}. Up to sixth order MIB schemes have been
constructed for the Poisson equation with ellipsoidal interfaces in
3D \cite{Yu:2007a}. There are two standing open problems concerning
high order elliptic interface schemes, i.e., the construction of
fourth-order 3D interface schemes for arbitrarily complex interfaces
with sharp geometric singularities and the construction of
sixth-order 3D interface schemes for arbitrarily curved smooth
interfaces \cite{Yu:2007a}.

Another challenging issue in elliptic interface problems arises from
nonsmooth interfaces or interfaces with Lipschitz continuity
\cite{Yu:2007,Yu:2007c,KLXia:2011,Hou:2010}.  Nonsmooth interfaces
are also referred to as geometric singularities, such as sharp
edges, cusps and tips, which commonly occur in real-world
applications. It is a challenge to design high order interface
schemes for geometric singularities both numerically and
analytically.  The first known second order accurate scheme for
nonsmooth interfaces was constructed in 2007 \cite{Yu:2007c}. Since
then, many other interesting second order schemes have been
constructed for this class of problems in 2D
\cite{TChen:2008,Hou:2010,Bedrossian:2010} and 3D \cite{Yu:2007a}.
The second order MIB method for 3D elliptic PDEs with arbitrarily
non-smooth interfaces or geometric singularities has found its
success in protein electrostatic analysis
\cite{Yu:2007,Geng:2007a,DuanChen:2011a}. However, it appears truly
challenging to develop fourth order schemes for arbitrarily shaped
nonsmooth interfaces in 3D domains, although fourth order schemes
have been reported for a few special interface geometries
\cite{Yu:2007a}. Due to the need  in practical applications, further
effort in this direction is expected.

The above mentioned difficulties in the solution of  interface
problems with geometric singularities significantly deteriorate in
certain physical situations. It is well-known that the electric
field diverges near the geometric singularities, such as tips of
electrodes,  antenna and elliptic cones, and sharp edges of planar
conductors. Since electric field is related to the gradient of the
electrostatic potential, i.e., the solution of the Poisson equation
in the electrostatic analysis, it turns out that the solution of the
elliptic equation has a lower regularity, e.g., the gradient does
not exist at the geometric singularities.  For isolated
singularities, one can alleviate the difficulty by  introducing an
algebraic factor to solve a regularized equation whose solution has
a high regularity \cite{Yu:2007,KLXia:2011}. In this manner, second
order MIB schemes have been constructed in the past
\cite{Yu:2007,KLXia:2011}. However, it is  also desirable to
directly solve the original PDEs with a low solution regularity. The
FEM developed by Hou {\it et al} is of  first order convergence in
the solution and 0.7th order convergence in the gradient of the
solution when the solution of the Poisson equation is $C^1$ or $H^2$
continuous and the interface is Lipschitz continuous
\cite{Hou:2010}. In fact, due to the low solution regularity induced
by the nonsmooth interface, it is very difficult to analyze the
convergence of numerical schemes because commonly used techniques
may no longer be available.  Therefore, there is a pressing need to
develop both high order numerical methods and rigorous analysis of
numerical methods for elliptic interface problems with low solution
regularities induced by geometric singularities.

The objective of the present work is to construct a new numerical
method for elliptic equations with low solution regularities induced
by nonsmooth interfaces. To this end, we employ a newly developed
weak Galerkin finite element method (WG-FEM) by Wang and Ye
\cite{JWang:2011}. Like the discontinuous Galerkin (DG) methods,
WG-FEM makes use of discontinuous functions in the finite element
procedure which  endows WG-FEMs with high flexibility to deal with
geometric complexities and boundary conditions. For interface
problems, such a flexibility gives rise to robustness in the
enforcement of interface jump conditions. Unlike DG methods, WG-FEM
enforces only weak continuity of variables naturally through well
defined discrete differential operators. Therefore, weak Galerkin
methods avoid pending parameters resulted from the excessive
flexibility given to individual elements. As a consequence, WG-FEMs
are absolutely stable once properly constructed. Recently, WG-FEMs
have been applied to the solution of second-order elliptic equations
\cite{LMu:2011a}, and the solution of Helmholtz equations with large
wave numbers \cite{LMu:2011b}. A major advantage of the present weak
formulation is that it naturally enables WG-FEM to handle interface
problems with low solution regularities. We demonstrate that the
present WG-FEM of lowest order is able to achieve from $1.75$th
order to second order of convergence in the solution and about first
order of convergence in the gradient when the solution of the
elliptic equation is $C^1$ or $H^2$ continuous and the interface is
$C^1$ or Lipschitz continuous.

The rest of this paper is organized as follows. Section
\ref{Sec:theory} is devoted to a description of the method and
algorithm. We shall design a weak Galerkin formulation for the
elliptic interface problem given in Eqs. (\ref{equ})-(\ref{ic2}).
Section \ref{Sec:Convergence} is devoted to a convergence analysis
for the weak Galerkin scheme presented in Section \ref{Sec:theory}.
In Section \ref{Sec:Numerical}, we shall present some numerical
results for several test cases in order to demonstrate the
performance of the proposed WG-FEM for elliptic interface problems
in 2D. We first consider a few popular benchmark test examples with
smooth but complex interfaces. We then carry out some investigation
about nonsmooth interfaces. Some of these test problems admit
solutions with low regularities, for which our numerical results
significantly improve the best known result in literature and are
better than our theoretical prediction. This paper ends with a
conclusion, especially a remark on the use and possible advantages
of high order WG-FEMs.

\section{Methods and Algorithms}\label{Sec:theory}

Let ${\cal T}_h$ be a partition of the domain $\Omega$ with mesh
size $h$. We require that the edges of the elements in ${\cal T}_h$
align with the interface $\Gamma$. Thus, the partition ${\cal T}_h$
can be grouped into two sets of elements denoted by ${\cal
T}_h^1={\cal T}_h\cap\Omega_1$ and ${\cal T}_h^2={\cal
T}_h\cap\Omega_2$, respectively. Observe that ${\cal T}_h^j$
provides a finite element partition for the subdomain $\Omega_j,
j=1,2$. The intersection of the partition ${\cal T}_h$ also
introduce a finite element partition for the interface $\Gamma$,
which shall be denoted by $\Gamma_h$. For simplicity, we adopt the
following notation:
$$
(v,w)_K:=\int_K vwdK, \qquad \langle v,w\rangle_{\partial
K}=\int_{\partial K}vwds.
$$

For each triangle $K\in {\cal T}_h$, let $K^0$ and $\partial K$
denote the interior and boundary of $K$ respectively. Denote by
$P_j(K^0)$ the set of polynomials in $K^0$ with degree no more than
$j$, and $P_\ell(e)$ the set of polynomials on each segment (edge or
face) $e$, $e\in\partial K$ with degree no more than $\ell$. A
discrete function $w=\{w_0,w_b\}$ refers to a polynomial with two
components in which the first component $w_0$ is associated with the
interior $K_0$ and $w_b$ is defined on each edge or face $e$,
$e\in\partial K$. Please note that $w_b$ may or may not equal to
$w_0$ on $\partial K$. Now we introduce three trial finite element
spaces as follows
\begin{eqnarray}
U_h:=&&\left\{ w=\{w_0, w_b\}:\ \{w_0, w_b\}|_{K}\in P_j(K^0)\times P_\ell(e),\;e\in\partial K, \forall K\in {\cal T}_h^1 \right\},\label{uh}\\
V_h:=&&\left\{ \rho=\{\rho_0, \rho_b\}:\ \{\rho_0, \rho_b\}|_{K}\in P_j(K^0)\times P_\ell(e),\;e\in\partial K, \forall K\in {\cal T}_h^2 \right\},\label{vh}\\
\Lambda_h:=&&\left\{\mu:\; \mu|_e\in P_m(e),\;e\in
\Gamma_h\right\}.\label{lh}
\end{eqnarray}
Define two test spaces by
\[
U_h^0=\{w=\{w_0, w_b\}\in U_h:\;w_b|_e=0,\;\mbox{for}\;e\in
\partial\Omega_1\setminus\Gamma\}
\]
and
\[
V_h^0=\{\rho=\{\rho_0, \rho_b\}\in
V_h:\;\rho_b|_e=0,\;\mbox{for}\;e\in
\partial\Omega_2\setminus\Gamma\}.
\]
For each $w=\{w_0, w_b\}\in U_h$ or $V_h$, we define the discrete
gradient of $w$, denoted by $\nabla_dw\in V_r(K)$ on each element
$K$, by the following equation:
\begin{equation}\label{dwg}
\int_K \nabla_{d} w\cdot q dK = -\int_K w_0 (\nabla\cdot q) dK+
\int_{\partial K} w_b (q\cdot\bn) ds,\qquad \forall q\in V_r(K),
\end{equation}
where $V_r(K)$ is a subspace of the set
of vector-valued polynomials of degree no more than $r$ on $K$.

The selection of the indices $j$, $\ell$, $m$, and  $r$ is critical
in the design of weak Galerkin finite element methods. A detailed
discussion on the selection of those indices can be found in
\cite{JWang:2011}. In the present study of the interface problem, we
shall let $j=\ell=m=k\ge 0$ in (\ref{uh})-(\ref{lh}) and chose the
Raviart-Thomas element for $V_r(K):=RT_k(K)$. These elements are
referred as $\left\{P_k(K^0)^2,\;P_k(e)^2,\;P_k(\Gamma)\right\}$
element in our numerical experiments. An exploration of other
possible combinations is left to interested readers for future
research. Recall that the Raviart-Thomas element
\cite{Raviart-Thomas} $RT_k(K)$ of order $k$ is of the following
form
\[
RT_k(K)=P_k(K)^2+\tilde{P}_k(K){\bf x},
\]
where $\tilde{P}_k(K)$ is the set of homogeneous polynomials of
degree $k$ and ${\bf x}=(x_1,x_2)$.

\medskip

\begin{algorithm}
A numerical approximation for (\ref{equ})-(\ref{ic2}) can be
obtained by seeking $u_h=\{u_0,u_b\}\in U_h$ satisfying
$u_b= Q_b g_1$ on $\partial\Omega_1\setminus\Gamma$, $v_h=\{v_0,v_b\}\in V_h$ satisfying
$v_b= Q_b g_2$ on $\partial\Omega_2\setminus\Gamma$ and $\lambda_h\in\Lambda_h$  such that
\begin{eqnarray}
(A\nabla_d u_h,\nabla_d w)-\langle\lambda_h,\;w_b\rangle_{\Gamma}&&= (f_1, w_0),\quad\qquad\quad\quad\quad \forall w\in U_h^0\label{wg1}\\
(A\nabla_d v_h,\nabla_d \rho)+\langle\lambda_h,\;\rho_b\rangle_{\Gamma}&&=
(f_2, \rho_0)+\langle\psi,\;\rho_b\rangle_{\Gamma},\;\quad \forall \rho\in V_h^0\label{wg2}\\
\langle u_b-v_b,\;\mu\rangle_\Gamma &&=\langle\phi,\;\mu\rangle_\Gamma,\quad\quad\quad\quad\quad\quad\forall\mu\in\Lambda_h.\label{wg3}
\end{eqnarray}
Here $Q_bg_1$ and $Q_bg_2$ are the standard $L^2$ projection of the
Dirichlet boundary data in $P_k(e)$ for any edge/face $e\in
\partial\Omega$.
\end{algorithm}

\section{Convergence Theory}\label{Sec:Convergence}

The goal of this section is to provide a convergence theory for the
weak Galerkin finite element method as described in the previous
section. First, we show that the WG algorithm has one and only one
solution in the corresponding finite element trial space.

\begin{lemma}\label{unique}
The weak Galerkin finite element method (\ref{wg1})-(\ref{wg3}) has
a unique solution.
\end{lemma}

\begin{proof}
It suffices to show that the solution of (\ref{wg1})-(\ref{wg3}) is
trivial when data is homogeneous; i.e.,
$f_1=f_2=g_1=g_2=\phi=\psi=0$. To this end, by letting $\mu=u_b-v_b$
in (\ref{wg3}) and then using the assumption of $\phi=0$ we obtain
$u_b=v_b$ on $\Gamma$. Next, by first letting $w=u_h$ in (\ref{wg1})
and $\rho=v_h$ in (\ref{wg2}) and then adding them up gives
\[
(A\nabla_d u_h,\nabla_d u_h)+(A\nabla_d v_h,\nabla_d v_h)=0,
\]
which implies $\nabla_du_h=0$ and $\nabla_dv_h=0$. Thus, both $u_h$
and $v_h$ are constants over $\Omega_1$ and $\Omega_2$,
respectively. It follows from the assumption of $u_b=0$ on the
non-empty set $\partial\Omega_1\setminus\Gamma$ that
$u_h=\{u_0,u_b\}=0$. Thus, we have $v_b=u_b=0$ on $\Gamma$, which
together with $\nabla_dv_h=0$ implies $v_h=0$. Finally, using
$\nabla_du_h=0$ and letting $w_b=\lambda_h$ on any edge/face
$e\in\Gamma_h$, we have from equation (\ref{wg1}) that
\[
\langle\lambda_h,\;\lambda_h\rangle_\Gamma=0,
\]
which implies $\lambda_h=0$.
\end{proof}

\medskip
Denote by $Q_h=\{Q_0,Q_b\}$ a local $L^2$ projection operator where:
\begin{eqnarray*}
Q_0:H^1(K^0)\rightarrow P_k(K^0),&&\quad
Q_b:H^{\frac12}(e)\rightarrow P_k(e),\;e\in\partial K
\end{eqnarray*}
are the usual $L^2$ projections into the corresponding spaces. Let
$\Pi_h$ be the usual projection operator in the mixed finite element
method such that $\Pi_h\bq\in H({\rm div},\Omega_i)$; and on each
$K\in {\cal T}_h^i$, one has $\Pi_h\bq \in RT_k(K)$ satisfying
$$
(\nabla\cdot\bq,\;w_0)_K=(\nabla\cdot\Pi_h\bq,\;w_0)_K, \qquad
\forall w_0\in P_k(K^0).
$$

\begin{lemma}
Let $\tau\in H({\rm div},\Omega_i)$ be a smooth vector-valued
function and $\Pi_h$ be the locally defined projection operator
commonly used in the mixed finite element method. Then, the
following identify holds true
\begin{equation}\label{key1}
\sum_{K\in {\cal T}_h^i}(-\nabla\cdot\tau, \;w_0)_K=\sum_{K\in {\cal
T}_h^i}(\Pi_h\tau, \;\nabla_d w)_K-\langle w_b,\; \tau\cdot
\bn\rangle_{\Gamma}
\end{equation}
for any $w=\{w_0,w_b\}\in V$ with $V=U_h^0$ or $V_h^0$ and $i=1,2$.
\end{lemma}
\begin{proof}
It follows from the definition of $\Pi_h$ and $\nabla_d$ that
\begin{eqnarray*}
\sum_{K\in {\cal T}_h^i}(-\nabla\cdot\tau, \;w_0)_K&=&\sum_{K\in {\cal T}_h^i}(-\nabla\cdot \Pi_h\tau, \;w_0)_K\\
&=&\sum_{K\in {\cal T}_h^i}\left((\Pi_h\tau,\; \nabla_d w)_K -
\langle w_b,\;
\Pi_h \tau\cdot \bn\rangle_{\partial K}\right). \\
\end{eqnarray*}
Now using the continuity of $\Pi_h\tau\cdot\bn$ across each interior
edge or face we arrive at
\begin{eqnarray*}
\sum_{K\in {\cal T}_h^i}(-\nabla\cdot\tau, \;w_0)_K &=&\sum_{K\in
{\cal T}_h^i}(\Pi_h \tau,\; \nabla_d w)_K-\langle w_b,\;
\Pi_h \tau\cdot \bn\rangle_{\partial\Omega_i}\\
&=&\sum_{K\in {\cal T}_h^i}(\Pi_h \tau,\; \nabla_d w)_K-\langle
w_b,\; \tau\cdot \bn\rangle_{\Gamma},
\end{eqnarray*}
where we have used the fact that $\langle w_b,\; \Pi_h\tau\cdot
\bn\rangle_{\Gamma} = \langle w_b,\; \tau\cdot \bn\rangle_{\Gamma}$
and $w_b=0$ on $\partial\Omega_i\backslash\Gamma$. This completes
the proof.
\end{proof}

\medskip
The following error estimates are straightforward from the
definition of $\Pi_h$ and $Q_h$. Readers can also find a
verification of the result in \cite{JWang:2011}.
\begin{lemma}
For $u\in H^{k+2}(\Omega_1)$ and $v\in H^{k+2}(\Omega_2)$ with $k\ge
0$, we have
\begin{eqnarray}
\|\Pi_h(A\nabla u)-A\nabla_d(Q_h u)\|&\le& Ch^{k+1}\|u\|_{k+2,\Omega_1},\label{approx}\\
\|\Pi_h(A\nabla v)-A\nabla_d(Q_h v)\|&\le&
Ch^{k+1}\|v\|_{k+2,\Omega_2}.\label{approx1}
\end{eqnarray}
\end{lemma}

It is well known that there exists a constant $C$ such that for any function $g\in
H^1(K)$
\begin{equation}\label{trace}
\|g\|_{e}^2 \leq C \left( h_K^{-1} \|g\|_K^2 + h_K
\|g\|_{1,K}^2\right),
\end{equation}
where $e$ is an edge of $K$.

\begin{theorem}
Let $(u_h, v_h,\lambda_h)\in U_h\times V_h\times \Lambda_h$ be the
solution arising from the weak Galerkin finite element scheme
(\ref{wg1})-(\ref{wg3}). Then, the following error estimates hold
true
\begin{eqnarray}
\|\nabla_d (Q_hu-u_h)\|+\|\nabla_d(Q_hv-v_h)\|&&\le Ch^{k+1}(\|u\|_{k+2}+\|v\|_{k+2}),\label{e-uv}\\
\|A\nabla u\cdot\bn-\lambda_h\|_\Gamma &&\le C
h^{k+\frac12}(\|u\|_{k+2}+\|v\|_{k+2}),\label{e-l}
\end{eqnarray}
where $C$ stands for a generic constant independent of the mesh size
$h$.
\end{theorem}

\begin{proof}
For simplicity, we assume that the coefficient $A$ is a constant
tensor on each element $K$. The proof can be extended to the general
case of non-constant tensor $A$ without any difficulty.

Now testing (\ref{equ}) with $w\in U_h^0$ and then using
(\ref{key1}) leads to
\[
(\Pi_h(A\nabla u),\;\nabla_d w)-\langle A\nabla u\cdot \bn,\;w_b\rangle_{\Gamma}=(f_1,\;w_0).
\]
Adding the term $(A\nabla_dQ_hu,\;\nabla_dw)$ to both sides of the
equation above, we obtain
\begin{eqnarray}
(A\nabla_dQ_hu,\;\nabla_dw)-\langle A\nabla u\cdot \bn,\;w_b\rangle_{\Gamma}&&=(f_1,\;w_0)\nonumber\\
&&-(\Pi_h(A\nabla u)-A\nabla_dQ_hu,\;\nabla_d w).\label{test1}
\end{eqnarray}
Similarly, on $\Omega_2$ we have for any $\rho\in V_h^0$ the
following relation
\begin{eqnarray}
(A\nabla_dQ_hv,\;\nabla_d\rho)-\langle A\nabla v\cdot \bn,\;\rho_b\rangle_{\Gamma}&&=(f_2,\;\rho_0)\nonumber\\
&&-(\Pi_h(A\nabla v)-A\nabla_dQ_hv,\;\nabla_d \rho). \label{test2}
\end{eqnarray}
Using the interface condition (\ref{ic2}), we from from
(\ref{test2}) that
\begin{eqnarray}
(A\nabla_dQ_hv,\;\nabla_d\rho)+\langle A\nabla u\cdot \bn,\;\rho_b\rangle_{\Gamma}&&
=(f_2,\;\rho_0)+\langle \psi,\;\rho_b\rangle_{\Gamma}\nonumber\\
&&-(\Pi_h(A\nabla v)-A\nabla_dQ_hv,\;\nabla_d \rho). \label{test2'}
\end{eqnarray}
Testing the interface condition (\ref{ic1}) by $\mu\in\Lambda_h$
implies
\begin{equation}\label{test3}
\langle Q_bu-Q_bv,\;\mu\rangle_\Gamma=\langle u-v,\;\mu\rangle_\Gamma=\langle \phi,\;\mu\rangle_\Gamma.
\end{equation}
For simplicity, we introduce the following notations to represent
the errors:
\begin{eqnarray*}
e_h^u\equiv\{e_0^u,\;e_b^u\}:&=&\{Q_0u-u_0,Q_bu-u_b\},\\
e_h^v\equiv\{e_0^v,\;e_b^v\}:&=&\{Q_0v-v_0,Q_bv-v_b\},\\
\epsilon_u:&=&Q_b(A\nabla u\cdot\bn)-\lambda_h.
\end{eqnarray*}
Thus, the difference of (\ref{test1}) and (\ref{wg1}) gives
\begin{eqnarray}
(A\nabla_de_h^u,\;\nabla_dw)-\langle \epsilon_u,\;w_b\rangle_{\Gamma}= -(\Pi_h(A\nabla u)-A\nabla_dQ_hu,\;\nabla_d w).\label{test11}
\end{eqnarray}
The difference of (\ref{test2'}) and (\ref{wg2}) gives
\begin{eqnarray}
(A\nabla_de_h^v,\;\nabla_d\rho)+\langle \epsilon_u,\;\rho_b\rangle_{\Gamma}= -(\Pi_h(A\nabla v)-A\nabla_dQ_hv,\;\nabla_d \rho).\label{test22}
\end{eqnarray}
The difference of (\ref{test3}) and (\ref{wg3}) gives
\begin{equation}\label{test33}
\langle e_b^u-e_b^v,\;\mu\rangle_\Gamma=0,\qquad\forall \mu\in
\Lambda_h.
\end{equation}
First letting $w=e_h^u$ and $\rho=e_h^v$ in (\ref{test11}) and
(\ref{test22}) and then adding them up gives
\begin{eqnarray*}
\|A^{\frac12}\nabla_de_h^u\|^2&&+\|A^{\frac12}\nabla_de_h^v\|^2=-(\Pi_h(A\nabla u)-A\nabla_dQ_hu,\;\nabla_d e_h^u)\\
&&-(\Pi_h(A\nabla v)-A\nabla_dQ_hv,\;\nabla_d e_h^v).
\end{eqnarray*}
Using (\ref{approx})-(\ref{approx1}), we have from the above
equation that
\begin{eqnarray*}
C(\|\nabla_de_h^u\|+\|\nabla_de_h^v\|)^2&&\le\|A^{\frac12}\nabla_de_h^u\|^2+\|A^{\frac12}\nabla_de_h^v\|^2\\
&&\le \|\Pi_h(A\nabla u)-A\nabla_dQ_hu\|\ \|\nabla_de_h^u\|\|\\
&& \ +\|\Pi_h(A\nabla v)-A\nabla_dQ_hv\|\ \|\nabla_de_h^v\|\\
&&\le Ch^{k+1}(\|u\|_{k+2}+\|v\|_{k+2}) (\|\nabla_de_h^u\|+\|\nabla_de_h^v\|),
\end{eqnarray*}
which implies
\begin{equation}\label{error1}
\|\nabla_de_h^u\|+\|\nabla_de_h^v\|\le Ch^{k+1}(\|u\|_{k+2}+\|v\|_{k+2}).
\end{equation}
Using (\ref{trace}) and the inverse inequality, we have for any edge
or face $e\subset\Gamma\cap\partial K$
\begin{eqnarray}
\|A\nabla_de_h^u\cdot\bn\|^2_e&&\le Ch^{-1}\|\nabla_de_h^u\|_K^2,\label{m1}\\
\|(\Pi_h(A\nabla u)-A\nabla_dQ_hu)\cdot\bn\|^2_e&&\le Ch^{-1}\|\Pi_h(A\nabla u)-A\nabla_dQ_hu\|_K^2.\label{m2}
\end{eqnarray}
Choosing $w=\{w_0,w_b\}$ in equation (\ref{test11}) such that $w_0=0$ for all $K\in {\cal T}_h^1$ and $w_b=\epsilon_u$ for $e\in\Gamma$ and $w_b=0$ otherwise  yields
\[
\|\epsilon_u\|^2_\Gamma=(A\nabla_de_h^u,\;\nabla_dw)+(\Pi_h(A\nabla u)-A\nabla_dQ_hu,\;\nabla_d w).
\]
Using (\ref{dwg}), the equation above becomes
\begin{eqnarray*}
\|\epsilon_u\|^2_\Gamma=\sum_{e\in\Gamma}(\langle A\nabla_de_h^u\cdot\bn,\;\epsilon_u\rangle_e+ \langle (\Pi_h(A\nabla u)-A\nabla_dQ_hu)\cdot\bn,\;\epsilon_u\rangle_e).
\end{eqnarray*}
Using (\ref{m1}), (\ref{m2}), (\ref{approx}) and (\ref{e-uv}), we
have
\begin{eqnarray*}
\|\epsilon_u\|^2_\Gamma &&\le C\sum_{e\in\Gamma}(\| A\nabla_de_h^u\cdot\bn\|_e\|\epsilon_u\|_e+ \|(\Pi_h(A\nabla u)-A\nabla_dQ_hu)\cdot\bn\|_e\|\epsilon_u\|_e)\\
&&\le C\left(\sum_{K\in {\cal T}_h^1}h^{-1}\|\nabla_de_h^u\|_K^2+
\sum_{K\in {\cal T}_h^1}h^{-1} \|\Pi_h(A\nabla u)-A\nabla_dQ_hu\|_K^2\right)^{1/2}\|\epsilon_u\|_\Gamma \\
&&\le Ch^{k+\frac12}(\|u\|_{k+2}+\|v\|_{k+2})\|\epsilon_u\|_\Gamma,
\end{eqnarray*}
which verifies the error estimate (\ref{e-l}). This completes the
proof of the theorem.
\end{proof}

\section{Numerical experiments}\label{Sec:Numerical}

The goal of this section is to numerically validate the proposed WG
algorithm by solving some benchmark elliptic interface problems for
which analytical solutions are known. To fully demonstrate the
accuracy and robustness of the WG method, we consider challenging
problems involving Lipschitz continuous interfaces, highly
oscillated solutions, and solutions with low regularities. For
simplicity, we use the piecewise constant finite element
$P_0(\partial K)-P_0(K)-RT_0(K)$ on structured triangular meshes in
all test cases. The mesh generation and computation are all
conducted in the MATLAB environment. Since the analytical solutions
are known for each test case, the nonhomogeneous terms $f_1$ and
$f_2$ of the elliptic equations and the Dirichlet boundary data can
be correspondingly derived. Moreover, the two interface jump
conditions for the solution and the flux across the interface are
calculated according to the given analytical solution. Numerical
errors in the solution and its gradient are reported in $L_\infty$
norms in all examples.

\begin{table}[!hb]
\caption{Numerical convergence test for Example 1.}
\label{tab.ex1}
\begin{center}
\begin{tabular}{||c||c|cc|cc||}
\hline\hline
Mesh & $\max \{h\}$ & \multicolumn{2}{c}{Solution}  & \multicolumn{2}{c||}{Gradient}   \\
\cline{3-4} \cline{5-6}
 & & $L_\infty$ error & order & $L_\infty$ error & order \\
\hline\hline
Level 1 &2.8553e-01   &1.0266e-02 &        &1.3042e-02 & \\ \hline
Level 2 &1.5110e-01   &2.9631e-03 & 1.9525 &6.5234e-03 & 1.0886 \\ \hline
Level 3 &7.7543e-02   &7.6763e-04 & 2.0247 &3.2615e-03 & 1.0391 \\ \hline
Level 4 &3.9258e-02   &1.9518e-04 & 2.0118 &1.6306e-03 & 1.0185 \\ \hline
Level 5 &1.9749e-02   &4.9198e-05 & 2.0058 &8.1523e-04 & 1.0090 \\
\hline\hline
\end{tabular}
\end{center}
\end{table}

\begin{figure}[!tb]
\centering
\begin{tabular}{cc}
  \resizebox{2.45in}{2.1in}{\includegraphics{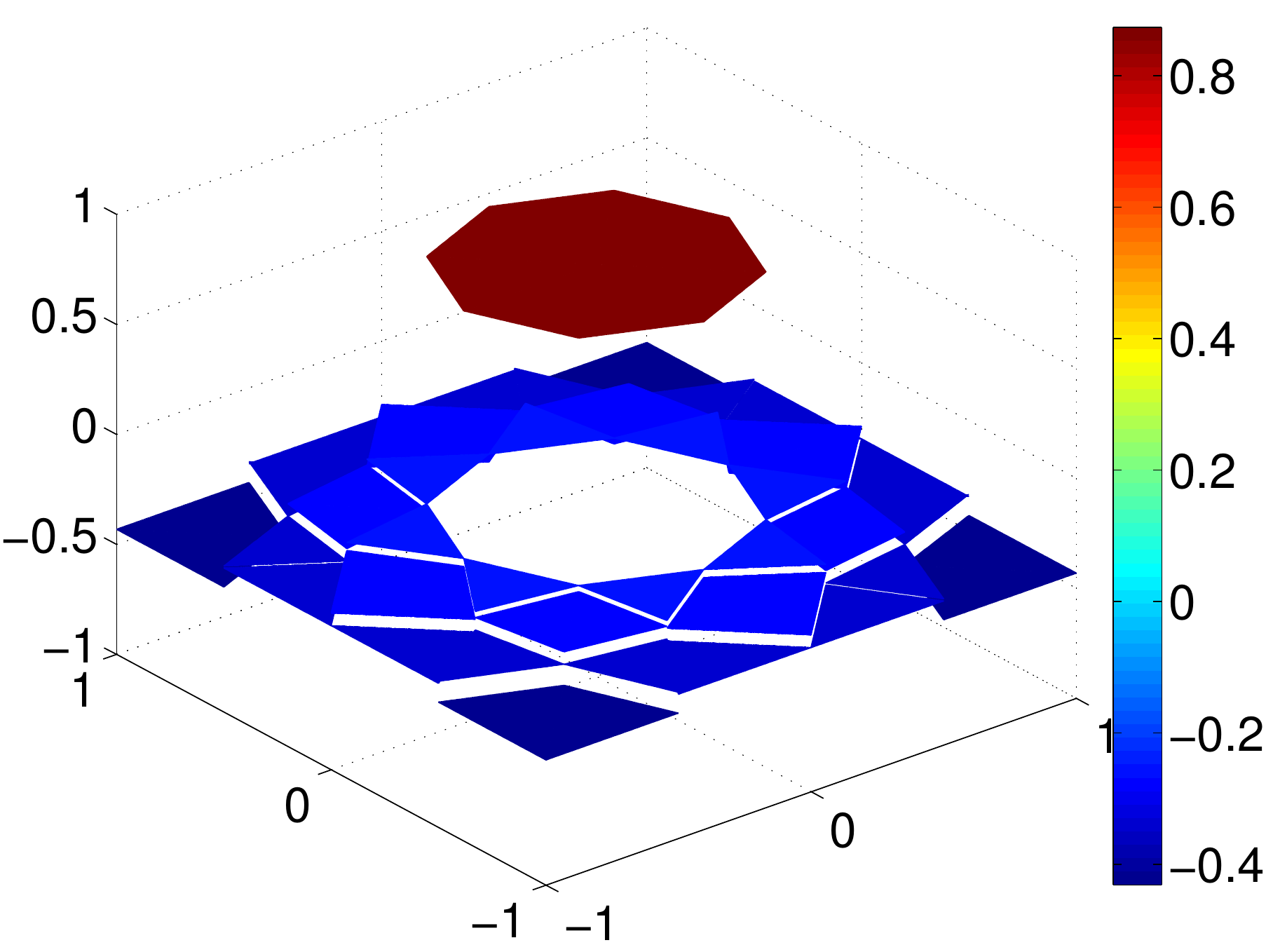}} \quad
  \resizebox{2.45in}{2.1in}{\includegraphics{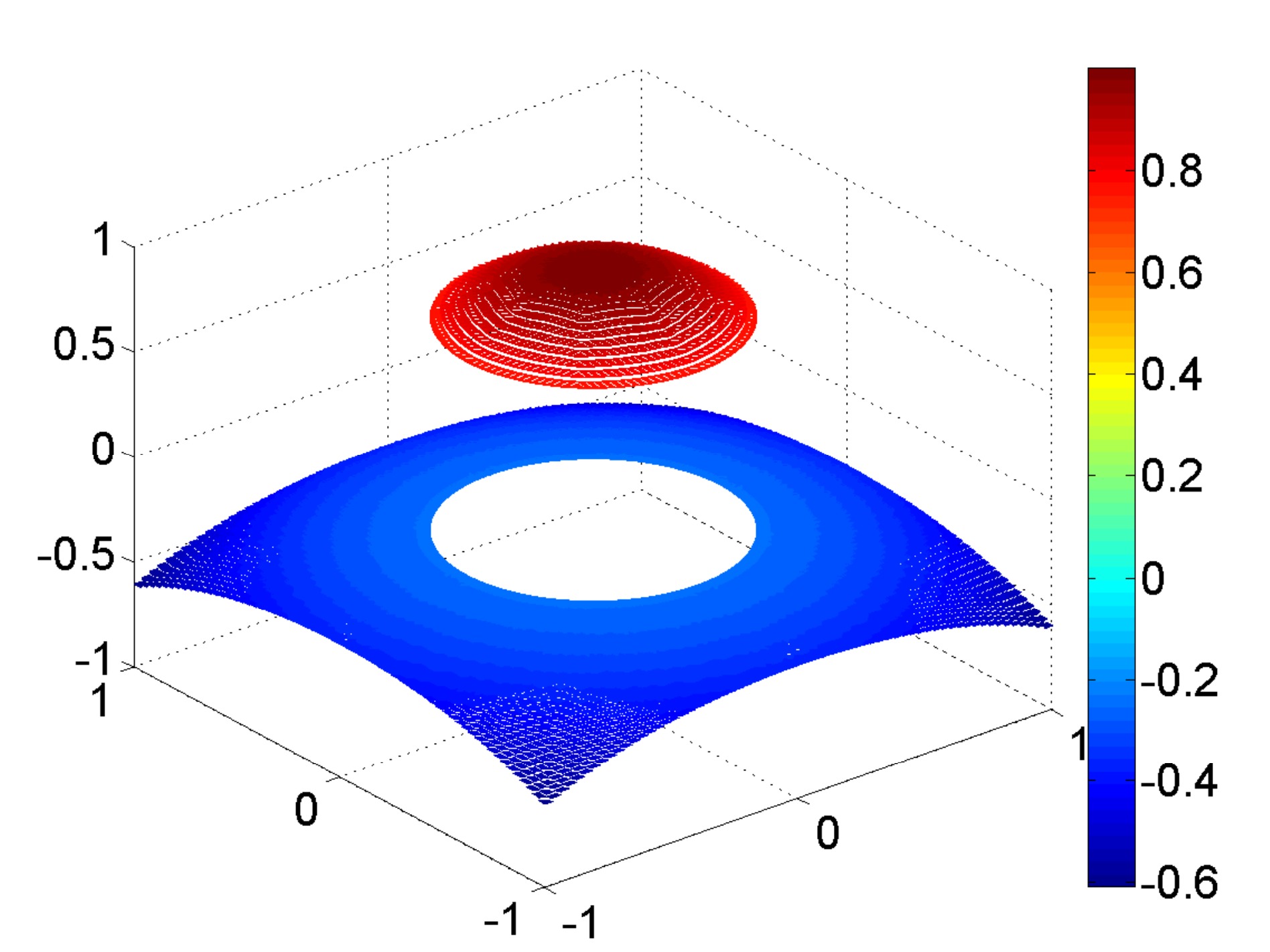}}
\end{tabular}
\caption{WG solutions of Example 1. Left: Mesh level 1; Right: mesh level 5.}
\label{fig.ex1}
\end{figure}

{\bf Example 1.} We first study a classical circular interface
problem \cite{Zhou:2006d}. Consider a square domain
$[-1,1]\times[-1,1]$ with a circular interface
$r^2=x^2+y^2=\frac{1}{4}$. The coefficient $A$ is defined to be
$A_1=b$ and $A_2=2$, respectively on each subdomain, for $r>0.5$ and
$r \le 0.5$. The analytical solution to the elliptic equation is
given as
\begin{align*}
u(x,y) &= -\left[\frac{1}{4}\left(1-\frac{1}{8b}-\frac{1}{b}\right)+\left(\frac{r^4}{2}+r^2\right)\right]/b
& r > 0.5 \\
v(x,y) & = -(x^2+y^2-1) &  r \le 0.5.
\end{align*}
By choosing $b=10$, the function and flux jumps
are actually constants \cite{Zhou:2006d}.

By using a uniform triangular mesh, the $L_\infty$ error for the
solution and its gradient of the WG method is reported in Table
\ref{tab.ex1}. Based on successive mesh refinements, the numerically
detected convergence rates are also reported for both error
measurements. It can be seen that the orders of convergence in
$L_\infty$ norm for the solution and gradient are, respectively, two
and one for piecewise constant WG finite element approximations.
This agrees with our theoretical results.

In Fig. \ref{fig.ex1}, the WG solutions based on mesh levels 1 and 5
are depicted. By using piecewise constant finite elements, the WG
solution at mesh level 1 clearly consists of piecewise constants. On
the other hand, a much better numerical solution is attained at mesh
level 5. It is also seen that the function and flux jumps are
constant across the circular interface.

\begin{table}[!hb]
\caption{Numerical convergence test for Example 2 with $\kappa=2$.}
\label{tab.ex2k2}
\begin{center}
\begin{tabular}{||c||c|cc|cc||}
\hline\hline
Mesh & $\max \{h\}$ & \multicolumn{2}{c}{Solution}  & \multicolumn{2}{c||}{Gradient}   \\
\cline{3-4} \cline{5-6}
 & & $L_\infty$ error & order & $L_\infty$ error & order \\
\hline\hline
Level 1 &3.1778e-01  &7.6628e-02 &        &1.5852e-01 & \\ \hline
Level 2 &1.5889e-01  &1.5119e-02 & 2.3415 &5.2258e-02 & 1.6009 \\ \hline
Level 3 &7.9444e-02  &3.6142e-03 & 2.0646 &2.0731e-02 & 1.3338 \\ \hline
Level 4 &3.9722e-02  &8.9375e-04 & 2.0157 &9.6505e-03 & 1.1031 \\ \hline
Level 5 &1.9861e-02  &2.2492e-04 & 1.9905 &4.7340e-03 & 1.0275 \\
\hline\hline
\end{tabular}
\end{center}
\end{table}

\begin{table}[!hb]
\caption{Numerical convergence test for Example 2 with $\kappa=8$.}
\label{tab.ex2k8}
\begin{center}
\begin{tabular}{||c||c|cc|cc||}
\hline\hline
Mesh & $\max \{h\}$ & \multicolumn{2}{c}{Solution}  & \multicolumn{2}{c||}{Gradient}   \\
\cline{3-4} \cline{5-6}
 & & $L_\infty$ error & order & $L_\infty$ error & order \\
\hline\hline
Level 1 & 1.5837e-01 &  8.6774e-01 &              & 5.2751e+00 &  \\ \hline
Level 2 & 7.9186e-02 &  1.4156e-01 & 2.6159  & 9.1695e-01  & 2.5243 \\ \hline
Level 3 & 3.9593e-02 &  2.3768e-02 & 2.5743  & 2.8679e-01  & 1.6768 \\ \hline
Level 4 & 1.9796e-02 &  5.8719e-03 & 2.0170  & 1.0732e-01  & 1.4180 \\ \hline
Level 5 & 9.8982e-03 &  1.4717e-03 & 1.9964  & 4.7255e-02  & 1.1834 \\
\hline\hline
\end{tabular}
\end{center}
\end{table}

\begin{figure}[!hb]
\centering
\begin{tabular}{cc}
  \resizebox{2.45in}{2.1in}{\includegraphics{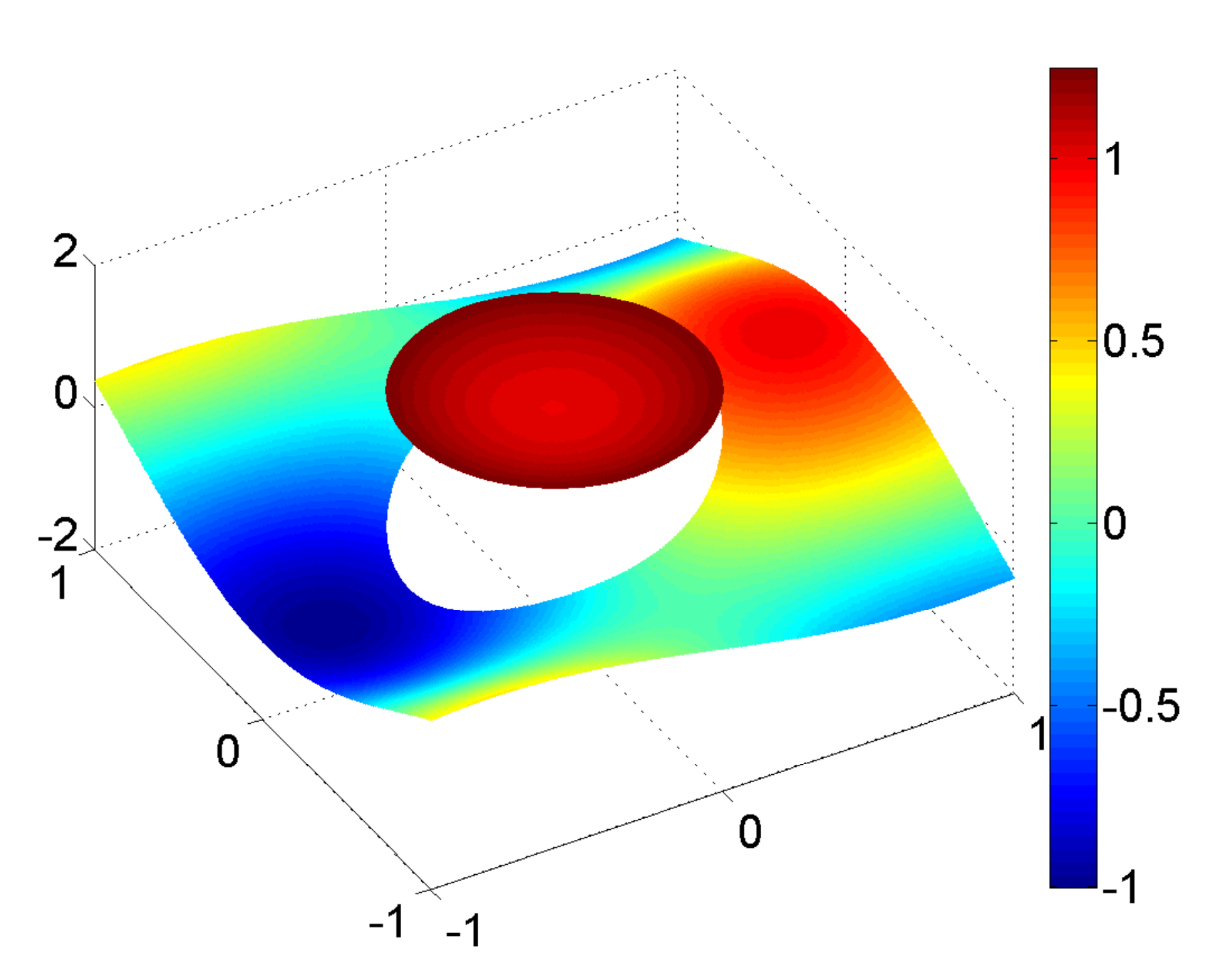}} \quad
  \resizebox{2.45in}{2.1in}{\includegraphics{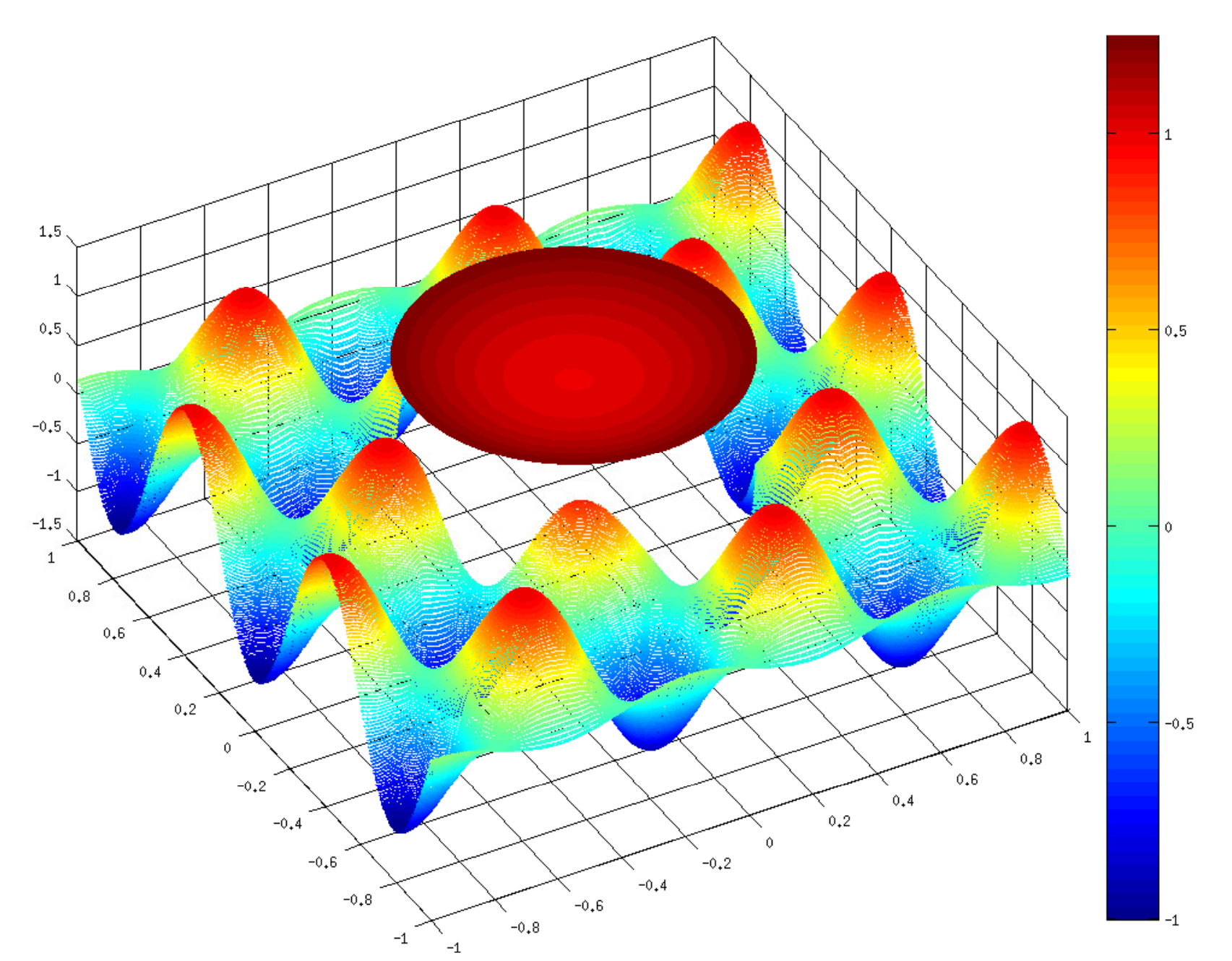}}
\end{tabular}
\caption{WG solutions at mesh level 5 for Example 2.
Left: $\kappa=2$; Right: $\kappa=8$.}
\label{fig.ex2}
\end{figure}

{\bf Example 2.} To further explore the potential of the proposed WG
method, we consider a circular interface problem with highly
oscillatory solutions. The high frequency or short wave solutions
are commonly encountered in solving electromagnetic wave propagation
and scattering problems governed by time domain Maxwell's equations
or frequency domain Helmholtz equations. Here, we consider the
following Helmholtz interface problem \cite{Zhou:2006d}
\begin{align}
-\nabla\cdot (A_1\nabla u) -k_1^2 u &=f_1 &  r > 0.5, \\
-\nabla\cdot (A_2\nabla v) -k_2^2 v &=f_2 &  r \le 0.5,
\end{align}
where the interface is still $r^2=\frac{1}{4}$ over the domain
$[-1,1]\times[-1,1]$. The wavenumber $k_{i}=\kappa \sigma_{i}$
($i=1,2$) depends on the wavenumber $\kappa$ of the wave solution in
free space and dielectric profiles $\sigma_{i}$. In the present
example, both $A$ and $\sigma_i$ are chosen to be discontinuous
across the interface. In particular $A_1=1$ and $\sigma_1=\sqrt{10}$
for $r >0.5$ and $A_2=10$ and $\sigma_2=1$ for $r \le 0.5$. The
analytical solution is given as
\begin{align*}
u(x,y) & = -\sin(\kappa x)\cos(\kappa y) &  r > 0.5, \\
v(x,y) & = -(x^2+y^2) & r \le 0.5.
\end{align*}
Other necessary interface jump conditions can be derived from the analytical
solution.

The WG solution of the Helmholtz equation with high wavenumbers has
been explored in \cite{LMu:2011b}. For the present Helmholtz
interface problem, a WG algorithm can be similarly constructed based
on equations (\ref{wg1}) - (\ref{wg3}). Two frequencies are tested,
i.e., $\kappa=2$ and $\kappa=8$, and the results are reported in
Tables \ref{tab.ex2k2} and \ref{tab.ex2k8}. By using $\kappa=8$, the
solution is of high frequency, see Fig. \ref{fig.ex2}. In comparison
with the numerical accuracy of the WG method for both cases, the WG
errors for $\kappa=8$ are larger even though a finer mesh is
employed in the same mesh level. Nevertheless, when we examine the
convergence rates, the WG method maintains the same rates, i.e.,
second and first orders, respectively, for the solution and its
gradient, in both cases. This verifies the capability of the WG
method in resolving short wave solutions.

\begin{table}[!hb]
\caption{Numerical convergence test for Example 3 with $b=10$.}
\label{tab.ex3b10}
\begin{center}
\begin{tabular}{||c||c|cc|cc||}
\hline\hline
Mesh & $\max \{h\}$ & \multicolumn{2}{c}{Solution}  & \multicolumn{2}{c||}{Gradient}   \\
\cline{3-4} \cline{5-6}
 & & $L_\infty$ error & order & $L_\infty$ error & order \\
\hline\hline
Level 1 &5.6522e-01  &1.4339e-01 &        &1.5956e-01 & \\ \hline
Level 2 &2.8261e-01  &2.6979e-02 & 2.4100 &9.1557e-02 & 0.8014 \\ \hline
Level 3 &1.4130e-01  &7.3332e-03 & 1.8792 &5.8167e-02 & 0.6544 \\ \hline
Level 4 &7.0652e-02  &1.4904e-03 & 2.2988 &2.9623e-02 & 0.9735 \\ \hline
Level 5 &3.5326e-02  &2.8124e-04 & 2.4058 &1.5250e-02 & 0.9579 \\
\hline\hline
\end{tabular}
\end{center}
\end{table}

\begin{figure}[!hb]
\centering
\begin{tabular}{cc}
  \resizebox{2.45in}{2.1in}{\includegraphics{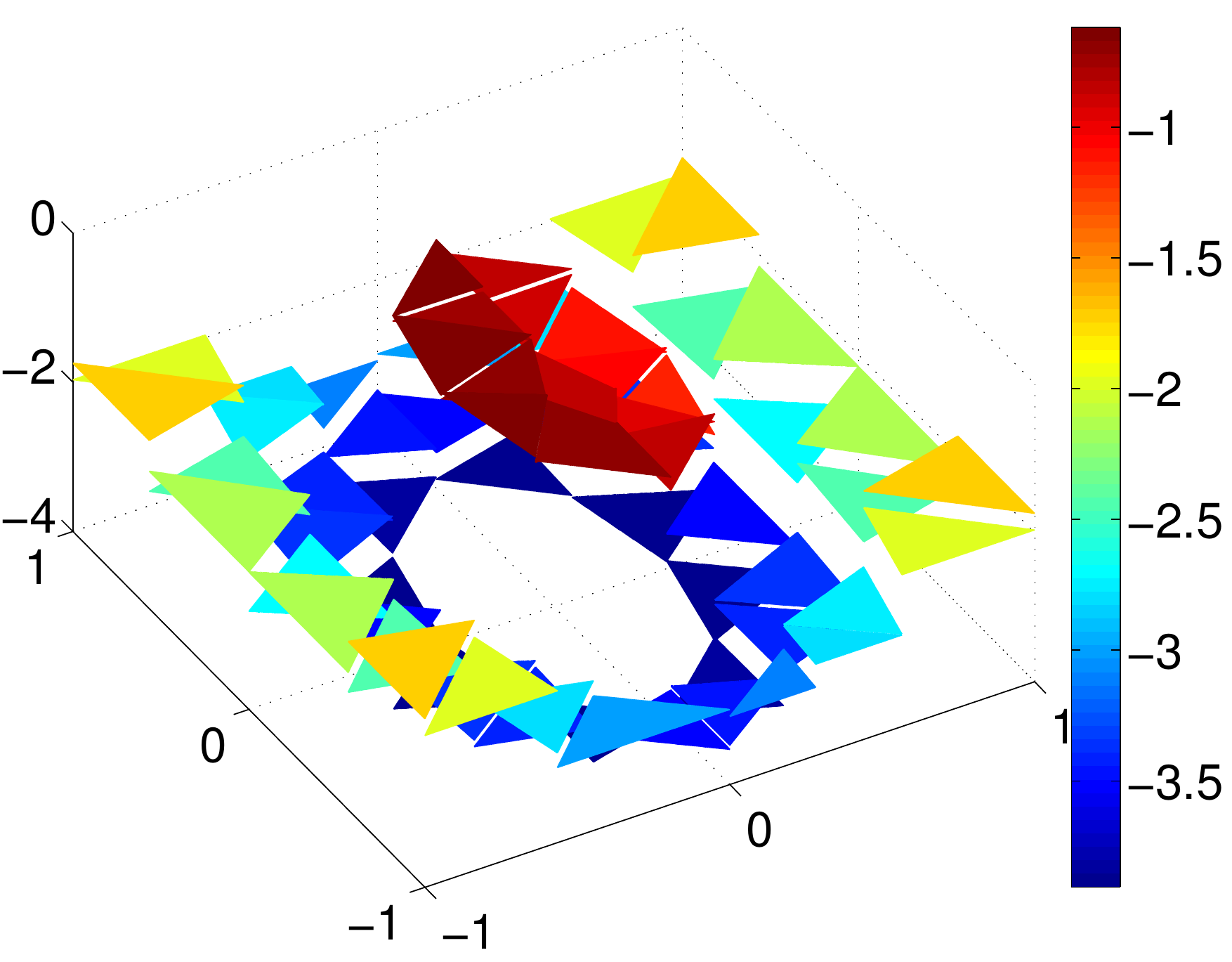}} \quad
  \resizebox{2.45in}{2.1in}{\includegraphics{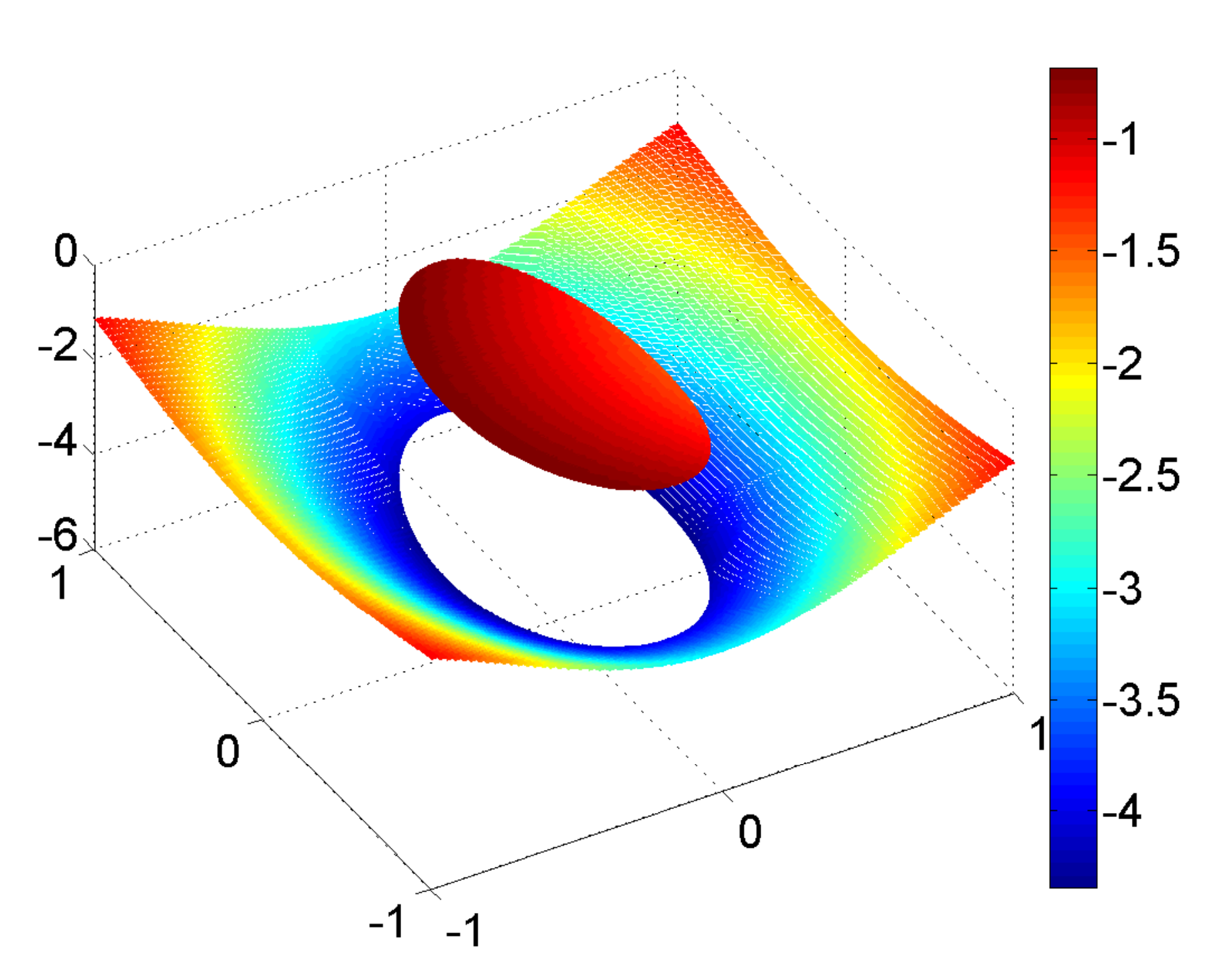}}
\end{tabular}
\caption{WG solutions of Example 3 with $b=10$.
Left: Mesh level 1; Right: mesh level 5.}
\label{fig.ex3b10}
\end{figure}

\begin{table}[!hb]
\caption{Numerical convergence test for Example 3 with $b=1000$.}
\label{tab.ex3b1000}
\begin{center}
\begin{tabular}{||c||c|cc|cc||}
\hline\hline
Mesh & $\max \{h\}$ & \multicolumn{2}{c}{Solution}  & \multicolumn{2}{c||}{Gradient}   \\
\cline{3-4} \cline{5-6}
 & & $L_\infty$ error & order & $L_\infty$ error & order \\
\hline\hline
Level 1 &5.6522e-01 &4.1482e-01 &        &3.1081e-01 &  \\ \hline
Level 2 &2.8261e-01 &5.3753e-02 & 2.9481 &1.0784e-01 & 1.5271 \\ \hline
Level 3 &1.4130e-01 &1.3707e-02 & 1.9713 &5.6088e-02 & 0.9431 \\ \hline
Level 4 &7.0652e-02 &3.0677e-03 & 2.1598 &2.9658e-02 & 0.9193 \\ \hline
Level 5 &3.5326e-02 &6.2057e-04 & 2.3055 &1.5204e-02 & 0.9640 \\
\hline\hline
\end{tabular}
\end{center}
\end{table}

\begin{figure}[!hb]
\centering
\begin{tabular}{cc}
  \resizebox{2.45in}{2.1in}{\includegraphics{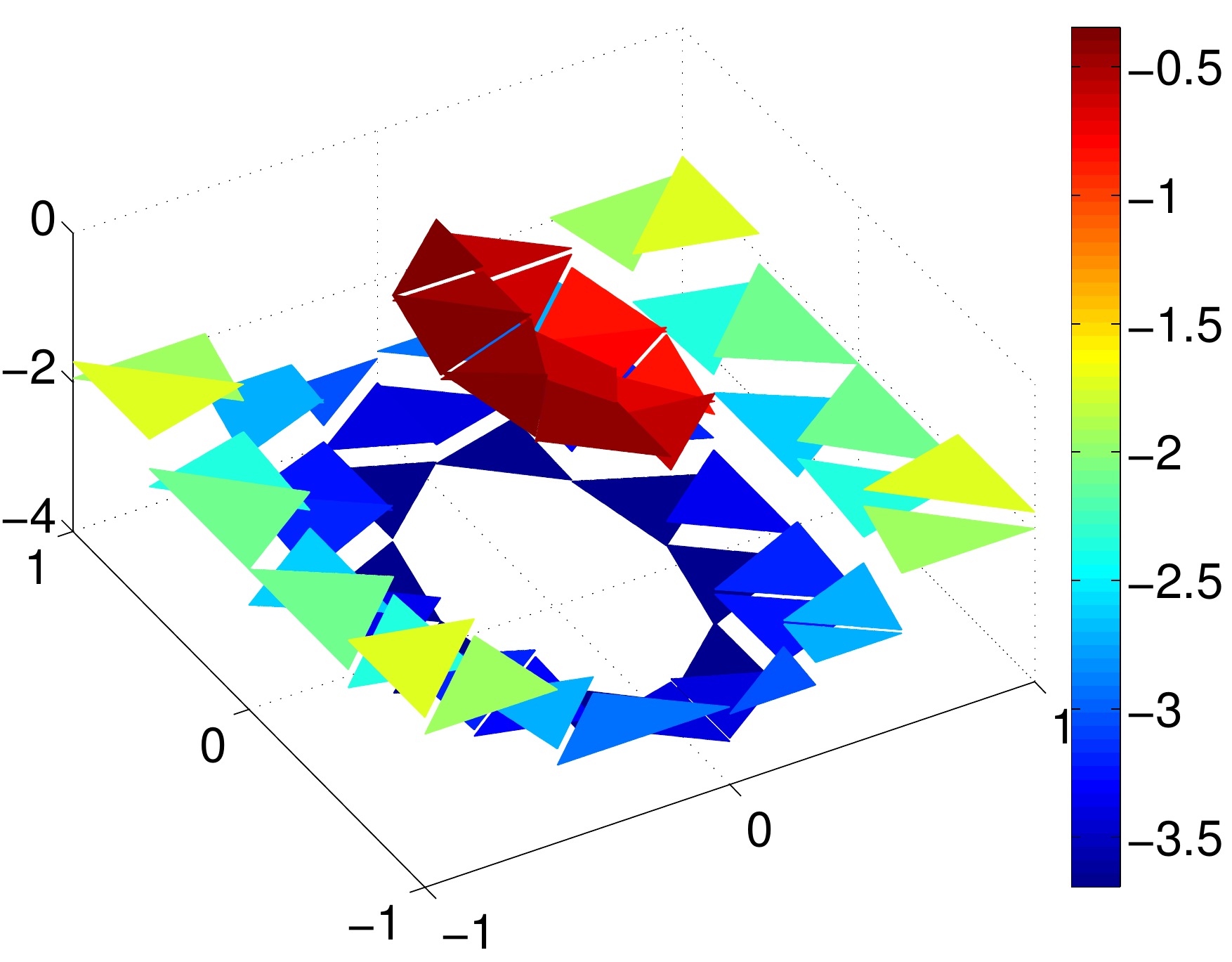}} \quad
  \resizebox{2.45in}{2.1in}{\includegraphics{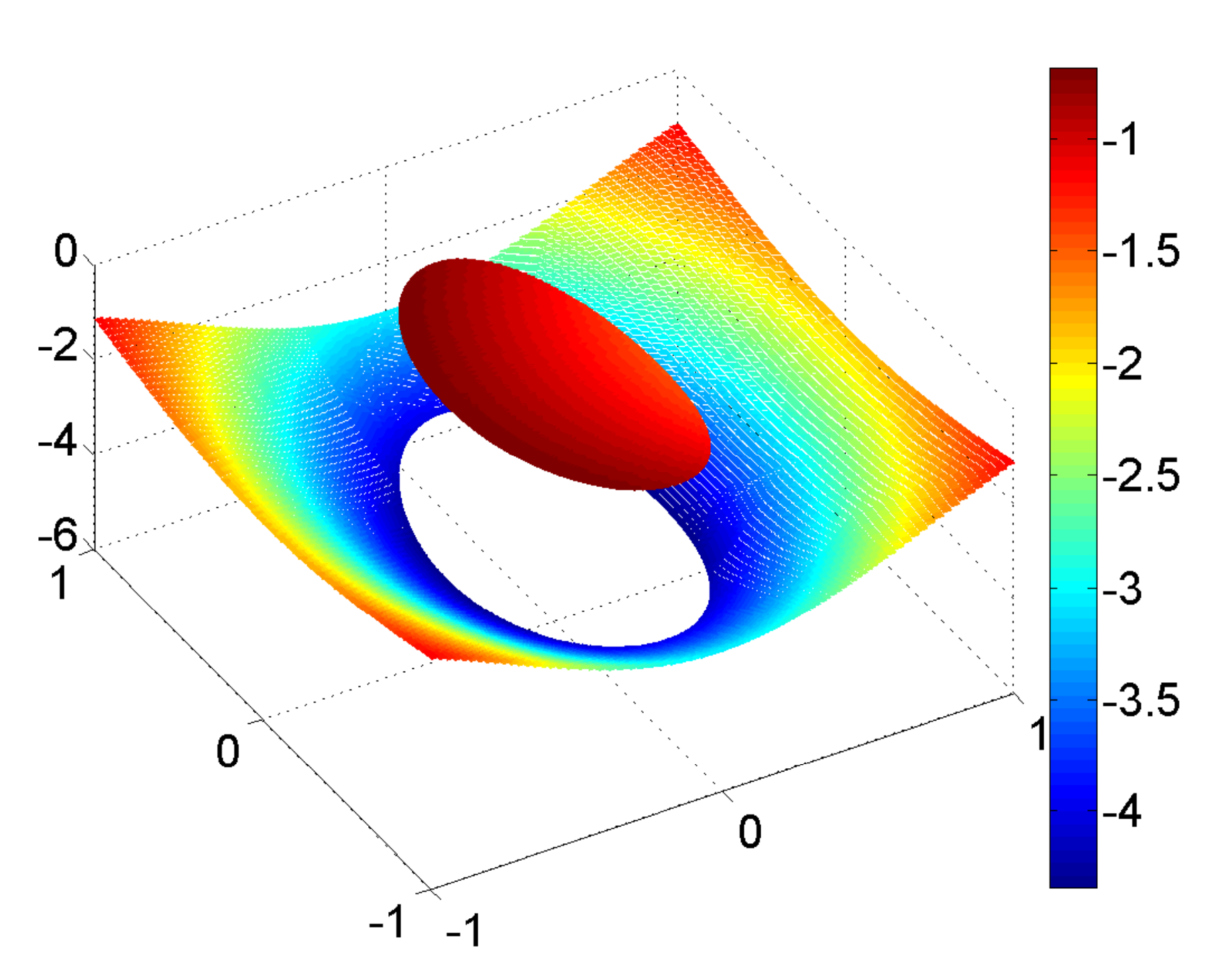}}
\end{tabular}
\caption{WG solutions of Example 3 with $b=1000$.
Left: Mesh level 1; Right: mesh level 5.}
\label{fig.ex3b1000}
\end{figure}

{\bf Example 3.} We next consider an example with a high contrast in
coefficient $A$ \cite{Zhou:2006d} to examine the robustness of the
WG method. Consider the second order elliptic equation over the
domain $[-1,1]\times[-1,1]$. The interface $\Gamma$ is defined to be
an ellipse
$$
\bigg(\frac{x}{10/27}\bigg)^2+\bigg(\frac{y}{18/27}\bigg)^2=1.
$$
We set $A_2=b$ inside of $\Gamma$, and $A_1=1$ outside of $\Gamma$.
The analytical solution is given as
\begin{align*}
u(x,y) & = 5e^{-x^2-y^2} & \mbox{outside } \Gamma, \\
v(x,y) & = e^x\cos(y) & \mbox{inside } \Gamma.
\end{align*}
Other necessary conditions can be derived from the analytical
solution.

Two values of $b$ are tested with $b=10$ and $b=1000$. The latter
case involves a much larger jump in $A$. This results in a larger
condition number for the corresponding discretized linear system.
This means that extensive computational time is usually needed in
the matrix solving. The WG method performs robustly for both cases.
The numerical results are presented in Table \ref{tab.ex3b10} and
Table \ref{tab.ex3b1000}, respectively, for $b=10$ and $b=1000$. It
can be seen in both tables that the convergence rates are first and
second orders, respectively, for the gradient and the solution in
$L_\infty$ norm. The WG solutions on mesh levels 1 and 5 are shown
in Fig. \ref{fig.ex3b10} and Fig. \ref{fig.ex3b1000}, respectively,
for $b=10$ and $b=1000$. We note that in both cases, the analytical
solutions are the same and are independent of the coefficient $A$.
Thus, even though  there are some minor differences in the WG
solutions at mesh level 1 for $b=10$ and $b=1000$, the WG solutions
are visually identical at mesh level 5. This demonstrates the
robustness of the WG method in handling interface problems with high
contrast.

\begin{table}[!hb]
\caption{Numerical convergence test for Example 4.}
\label{tab.ex4}
\begin{center}
\begin{tabular}{||c||c|cc|cc||}
\hline\hline
Mesh & $\max \{h\}$ & \multicolumn{2}{c}{Solution}  & \multicolumn{2}{c||}{Gradient}   \\
\cline{3-4} \cline{5-6}
 & & $L_\infty$ error & order & $L_\infty$ error & order \\
\hline\hline
Level 1 &3.4726e-01  &1.5885e-03 &        &3.5260e-02 & \\ \hline
Level 2 &1.7363e-01  &3.8053e-04 & 2.0616 &1.6630e-02 & 1.0842 \\ \hline
Level 3 &8.9010e-02  &9.1082e-05 & 2.1399 &9.0379e-03 & 0.9126 \\ \hline
Level 4 &4.9602e-02  &2.6778e-05 & 2.0936 &5.0265e-03 & 1.0034 \\ \hline
Level 5 &2.4286e-02  &5.9847e-06 & 2.0982 &2.4459e-03 & 1.0087 \\
\hline\hline
\end{tabular}
\end{center}
\end{table}


\begin{figure}[!hb]
\centering
\begin{tabular}{cc}
  \resizebox{2.45in}{2.1in}{\includegraphics{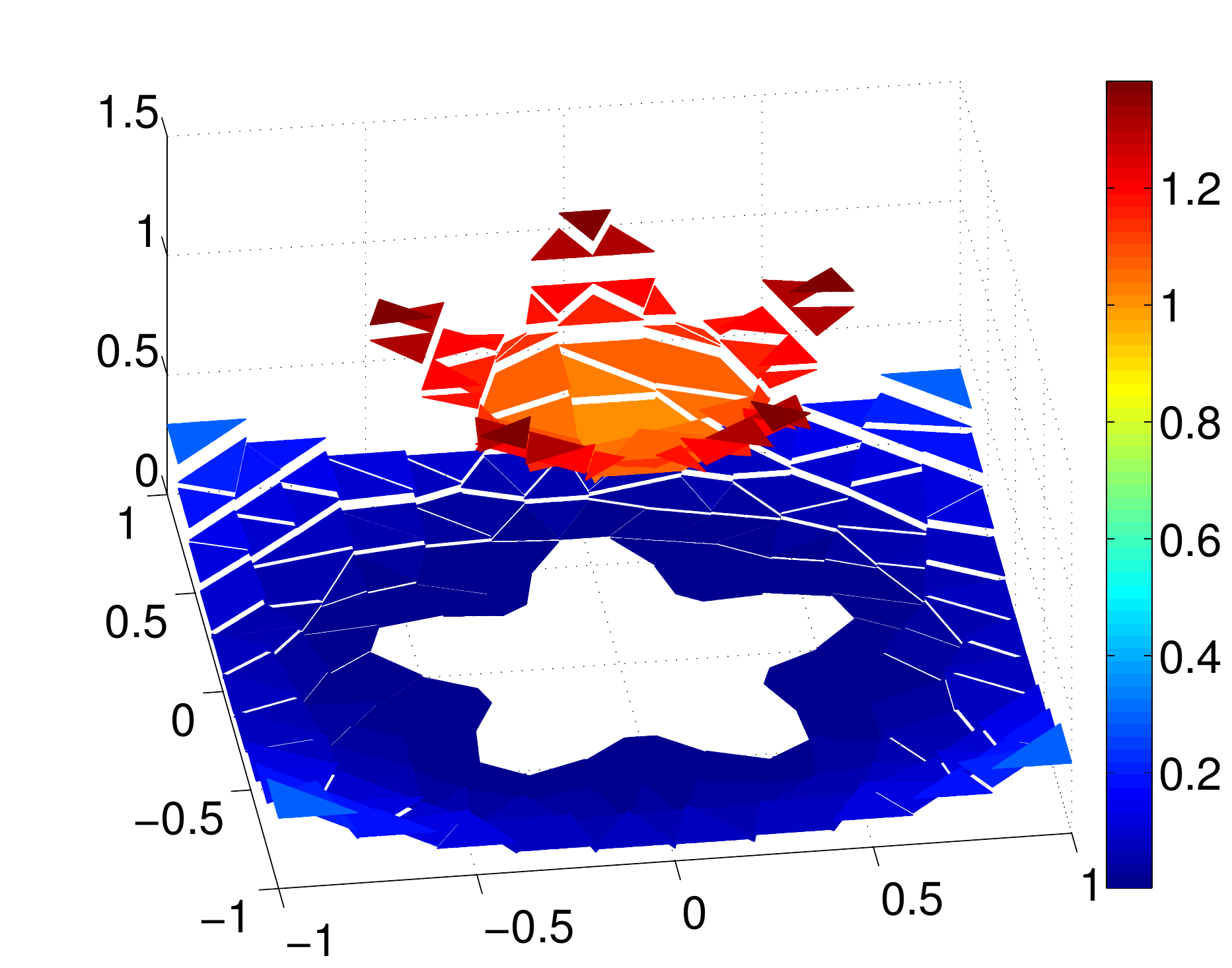}} \quad
  \resizebox{2.45in}{2.1in}{\includegraphics{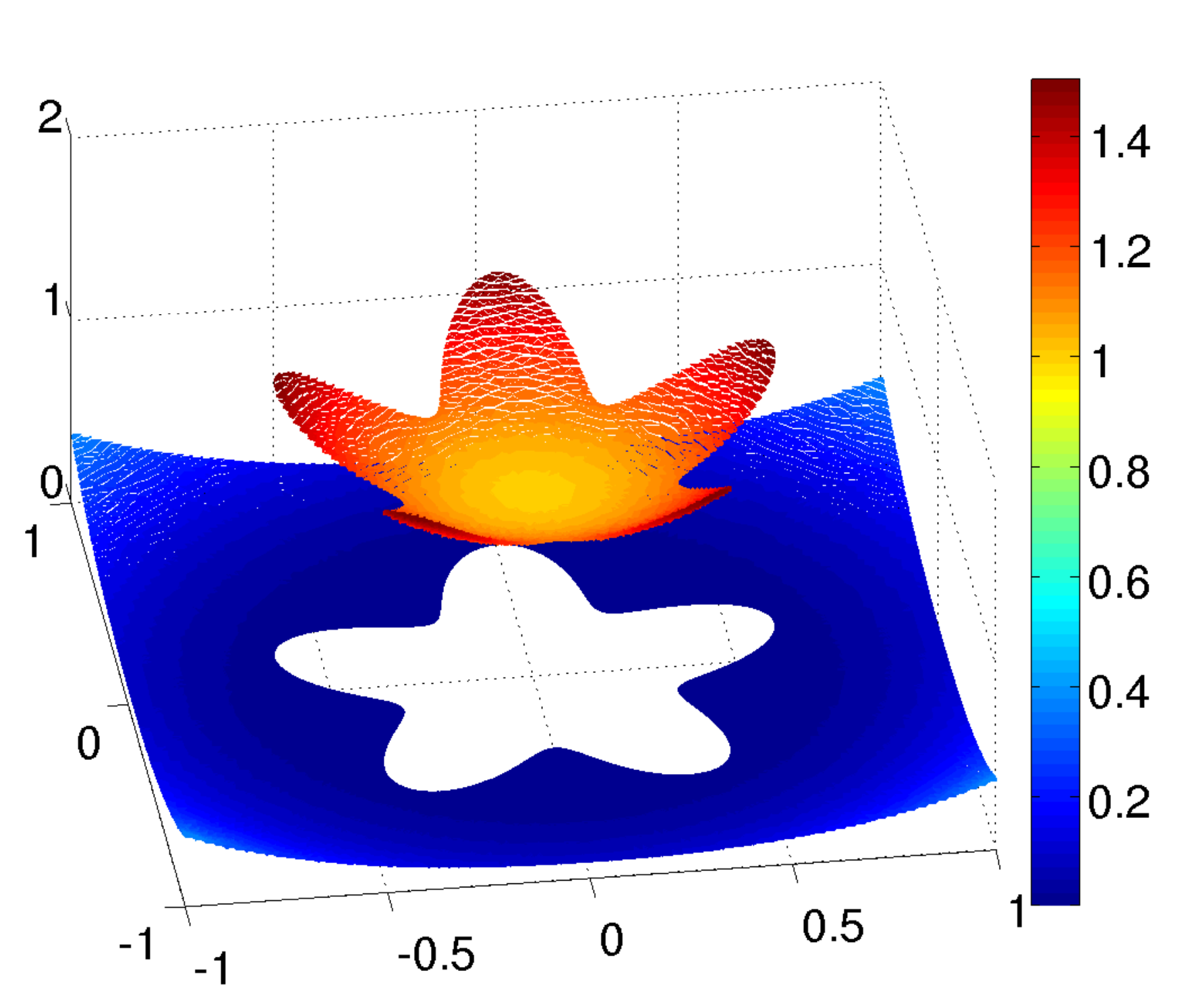}}
\end{tabular}
\caption{WG solutions of Example 4.
Left: Mesh level 1; Right: mesh level 5.}
\label{fig.ex4}
\end{figure}

\begin{figure}[!hb]
\centering
\begin{tabular}{cc}
\resizebox{2.5in}{2.5in}{\includegraphics{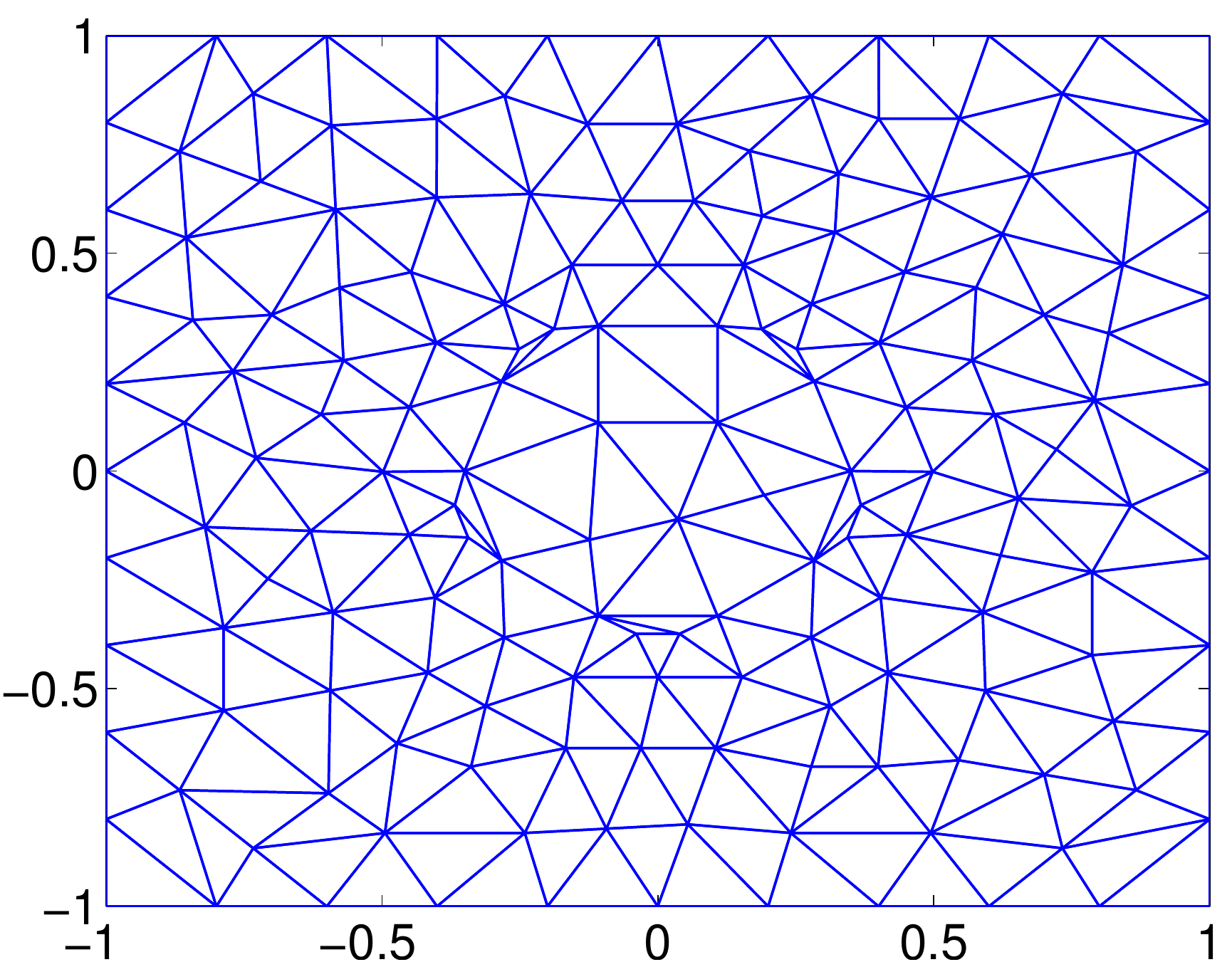}} \quad
\resizebox{2.5in}{2.5in}{\includegraphics{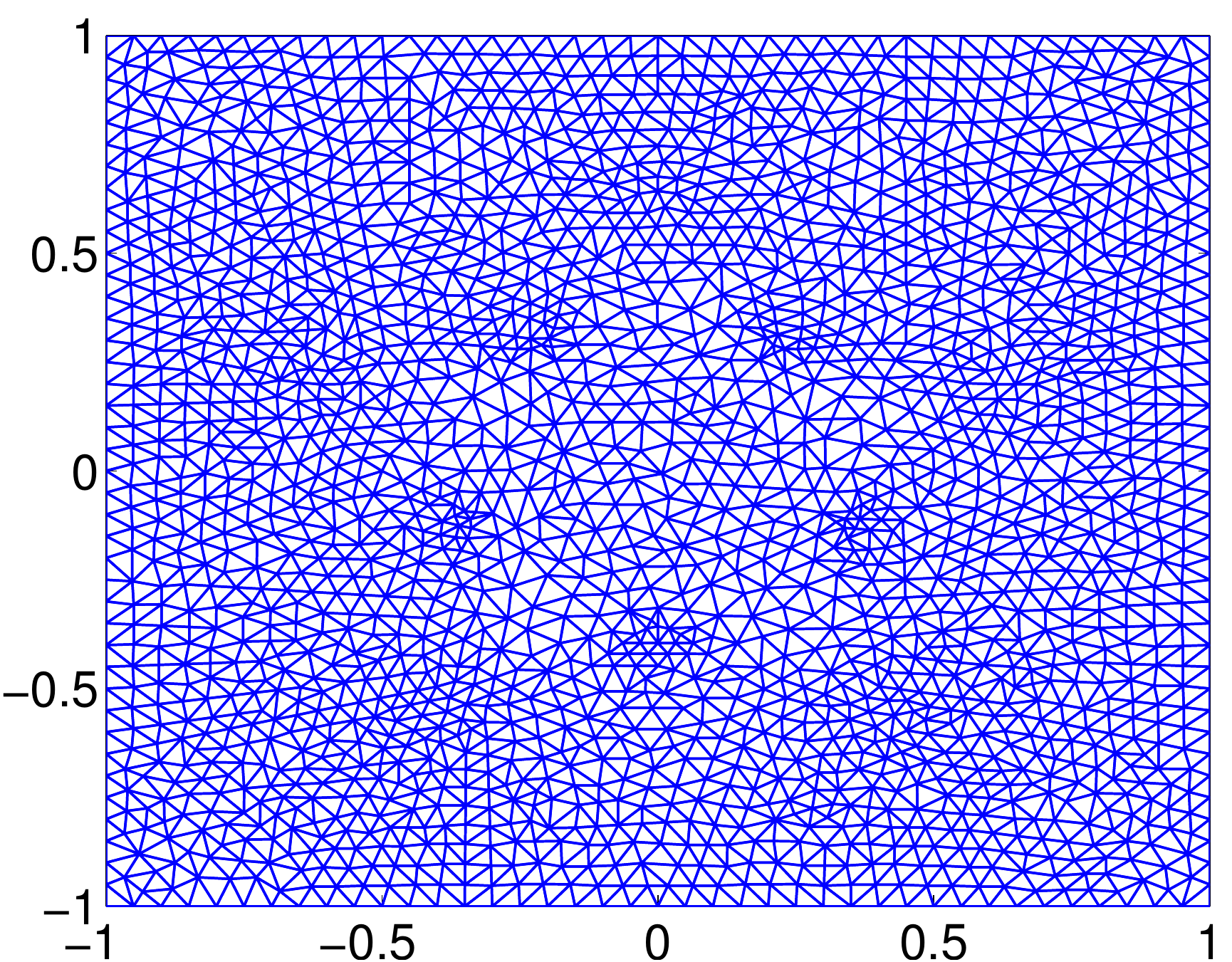}} \\
\resizebox{4in}{4in}{\includegraphics{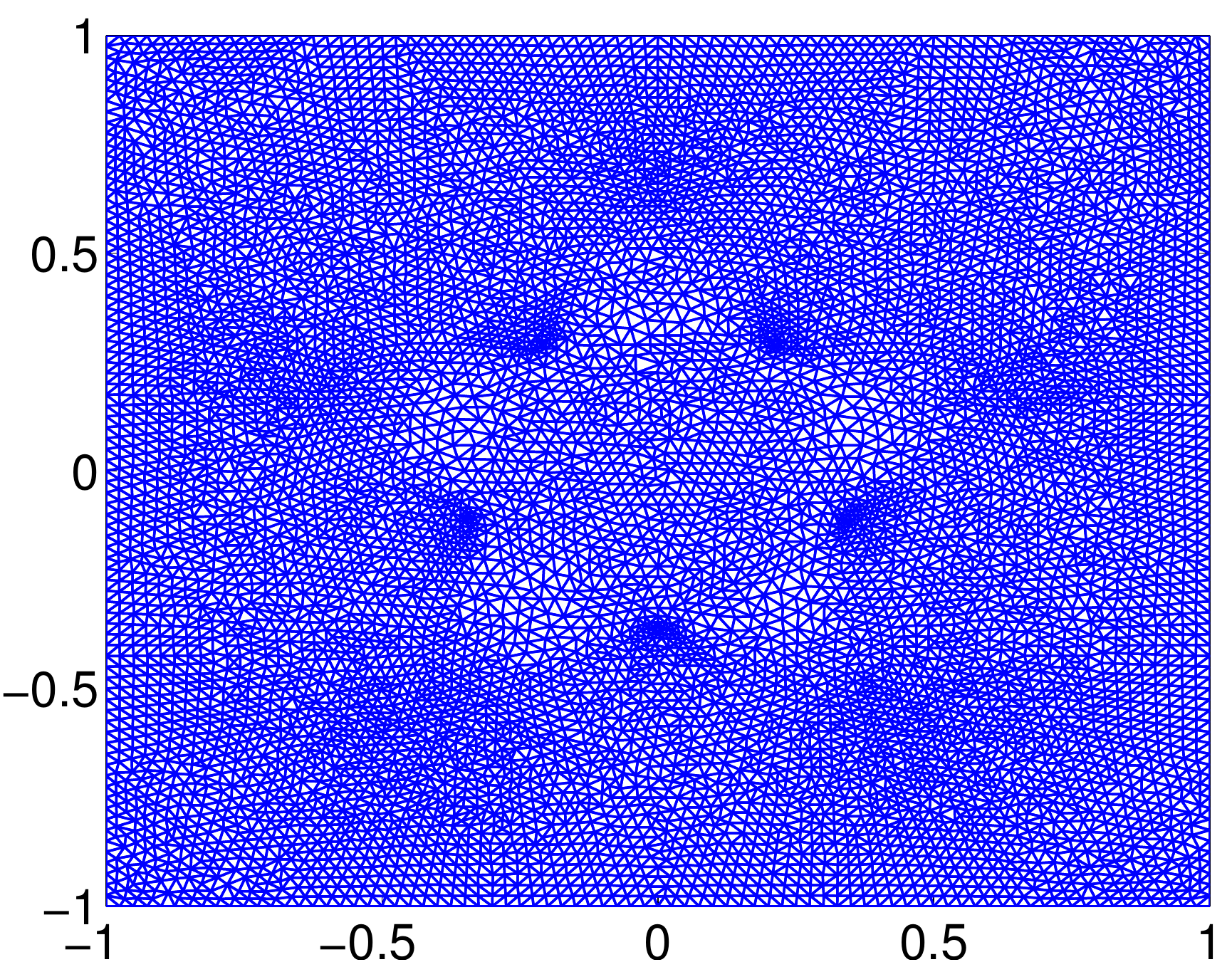}}
\end{tabular}
\caption{Finite element meshes of Example 4.
Top left: Mesh level 1; Top right: Mesh level 3; Bottom: Mesh level 5.}
\label{fig.ex4mesh}
\end{figure}

{\bf Example 4.}
We next consider a classical elliptic interface problem with both concave
and convex curve segments \cite{Zhou:2006d}.
The interface $\Gamma$ is parametrized with the polar angle $\theta$
$$r=\frac{1}{2}+\frac{\sin(5\theta)}{7},$$
inside a square domain $[-1,1]\times[-1,1]$.
The coefficient is chosen to be $A_1=10$ and $A_2=1$, respectively,
for outside and inside $\Gamma$.
The analytical solution is given as
\begin{align*}
u(x,y) & = 0.1(x^2+y^2)^2-0.01\ln(2\sqrt{x^2+y^2}) & \mbox{outside } \Gamma \\
v(x,y) & =e^{x^2+y^2} & \mbox{inside } \Gamma.
\end{align*}
Other necessary interface jump conditions can be derived from the analytical solutions.

The WG solutions and their convergence are reported in Table
\ref{tab.ex4} and Fig. \ref{fig.ex4}, respectively. It can be seen
from the table that the convergence is of order two and one for the
solution and its gradient, respectively, in the $L_\infty$ norm. It
should be pointed out that mesh quality plays an important role in
the performance of general finite element methods. For weak
Galerkin, we noted that the quality of mesh affects the
corresponding numerical approximation in $L_\infty$ norms. For
example, we noted a slower convergence when the finite element
partition consists of both acute and non-acute triangles with
nontrivial portion. After eliminating most of the non-acute
triangles, the WG approximation exhibits an improved rate of
convergence numerically. Figure \ref{fig.ex4mesh} demonstrates some
meshes generated by MATLAB in the present study after an elimination
of most non-acute triangles through a numerical procedure. It can be
seen that, at each level, the triangular mesh is fitted to the
interface at the nodal points. Observe that these meshes are highly
non-uniform and unstructured. In particular, it is clear that much
smaller triangles are employed to resolve the concave portion of the
interface. This somehow helped the WG method to deliver numerical
solutions with high accuracy since the solution changes rapidly in
regions with large curvatures.

\begin{table}[!hb]
\caption{Numerical convergence test for Example 5.}
\label{tab.ex5}
\begin{center}
\begin{tabular}{||c||c|cc|cc||}
\hline\hline
Mesh & $\max \{h\}$ & \multicolumn{2}{c}{Solution}  & \multicolumn{2}{c||}{Gradient}   \\
\cline{3-4} \cline{5-6}
 & & $L_\infty$ error & order & $L_\infty$ error & order \\
\hline\hline
Level 1  & 6.2575e-01  & 1.5275e-02 &             & 4.8833e-02 &  \\ \hline
Level 2  & 2.9662e-01  & 3.9259e-03 & 1.8200 & 3.1059e-02 & 0.6062  \\ \hline
Level 3  & 1.5170e-01  & 1.0586e-03 & 1.9546 & 1.5584e-02 & 1.0285  \\ \hline
Level 4  & 7.9102e-02  & 3.3419e-04 & 1.7707 & 8.0209e-03 & 1.0200  \\ \hline
Level 5  & 4.4180e-02  & 1.1698e-04 & 1.8022 & 4.4282e-03 & 1.0199  \\
\hline\hline
\end{tabular}
\end{center}
\end{table}

\begin{figure}[!hb]
\centering
\begin{tabular}{cc}
  \resizebox{2.45in}{2.1in}{\includegraphics{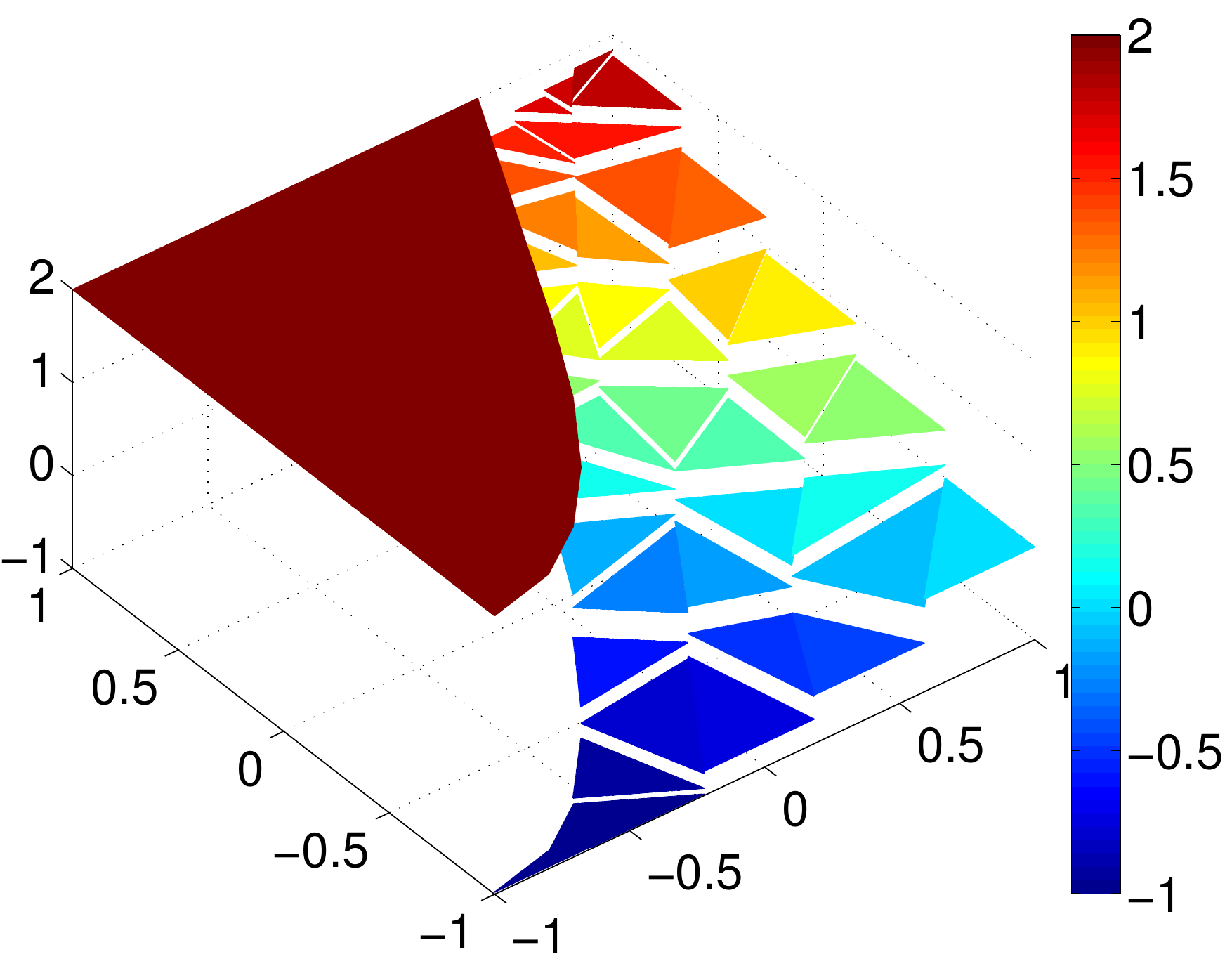}} \quad
  \resizebox{2.45in}{2.1in}{\includegraphics{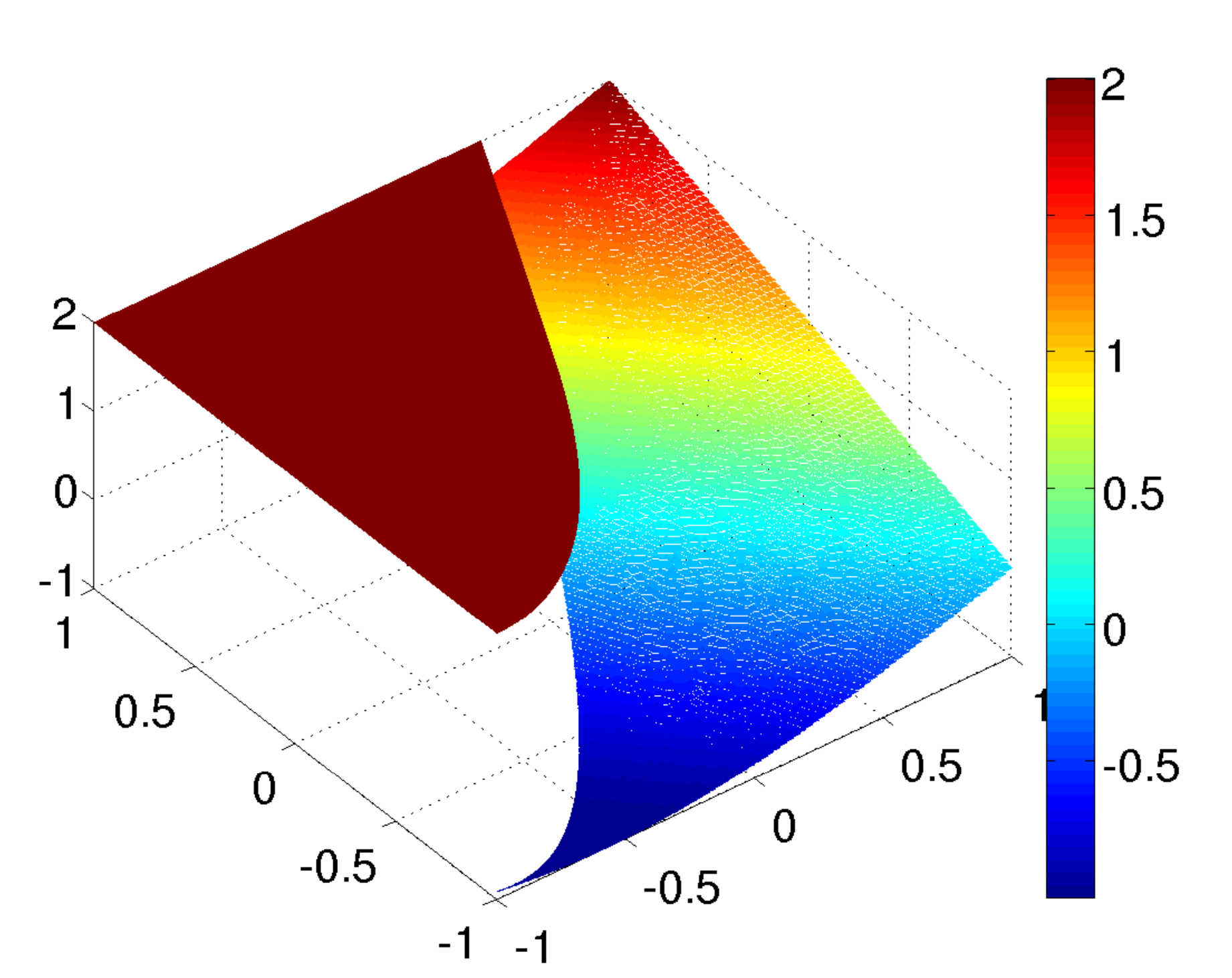}}
\end{tabular}
\caption{WG solutions of Example 5.
Left: Mesh level 1; Right: mesh level 5.}
\label{fig.ex5}
\end{figure}

\bigskip

{\bf Example 5.} In this and next two examples, we consider an
interface problem with $C^1$ continuous interfaces as discussed in
\cite{Hou:2010}. The domain is given as $\Omega=[-1,1]\times[-1,1]$,
and the interface $\Gamma$ is defined with two pieces: $y=2x$ for
$x+y>0$ and $y=2x+x^2$, for $x+y\le 0$. The interface $\Gamma$
divides the domain into two parts $\Omega_1$ and $\Omega_2$. Here we
denote $\Omega_1$ and $\Omega_2$, respectively, to be the part on
the left and right of $\Gamma$. An inhomogeneous second order
elliptic equation is considered, i.e., the coefficient $A$ is not
piecewise constant. Instead, we have $A_1(x,y)=(xy+2)/5$ in
$\Omega_1$ and $A_2(x,y)=(x^2-y^2+3)/7$ in $\Omega_2$. The
analytical solution $u$ is fixed to be $u=2$ in $\Omega_1$, while in
$\Omega_2$, $v$ is defined piecewisely by
$$
v(x,y)= \begin{cases}
\sin(x+y),  \quad \mbox{ if }x+y\le 0\\
x+y,     \quad \quad \quad         \mbox{ if }x+y>0.
\end{cases}
$$
Other necessary conditions can be derived from the analytical form
of the solution. In particular, the Dirichlet boundary data is
derived for both $u=g$ and $v=g$, since in this example, $\Omega_2$
is not enclosed by $\Omega_1$.

A numerical investigation was conducted and the results are reported
in Table \ref{tab.ex5}. Again, the WG method attains the theoretical
order of convergence for this example, i.e., first and second order
of accuracy respectively for the gradient and the solution in
$L_\infty$ norm. We note that due to the piecewise definition, the
solution $v(x,y)$ is $C^2$ continuous but not $C^3$ across the line
$x+y=1$. In our computation, the line $x+y=1$ is not treated as a
boundary. Thus, our triangulation is not aligned with this line.
This can be seen clearly in the graph of the WG solution on mesh
level 1 (Fig. \ref{fig.ex5} left). In the right chart of Fig.
\ref{fig.ex5}, a smooth transition can be seen across $x+y=1$ in the
WG solution on mesh level 5.

\begin{table}[!hb]
\caption{Numerical convergence test for Example 6.}
\label{tab.ex6}
\begin{center}
\begin{tabular}{||c||c|cc|cc||}
\hline\hline
Mesh & $\max \{h\}$ & \multicolumn{2}{c}{Solution}  & \multicolumn{2}{c||}{Gradient}   \\
\cline{3-4} \cline{5-6}
 & & $L_\infty$ error & order & $L_\infty$ error & order \\
\hline\hline
Level 1  & 6.2575e-01 & 2.6559e-02 &             & 4.9273e-02 & \\ \hline
Level 2  & 2.9662e-01 & 5.9674e-03 & 2.0001 & 2.8456e-02 & 0.7355 \\ \hline
Level 3  & 1.5170e-01 & 1.5342e-03 & 2.0257 & 1.4493e-02 & 1.0062 \\ \hline
Level 4  & 7.9102e-02 & 4.4127e-04 & 1.9137 & 7.5081e-03 & 1.0100 \\ \hline
Level 5  & 4.4180e-02 & 1.5022e-04 & 1.8500 & 3.9584e-03 & 1.0990 \\
\hline\hline
\end{tabular}
\end{center}
\end{table}

\begin{figure}[!hb]
\centering
\begin{tabular}{cc}
  \resizebox{2.45in}{2.1in}{\includegraphics{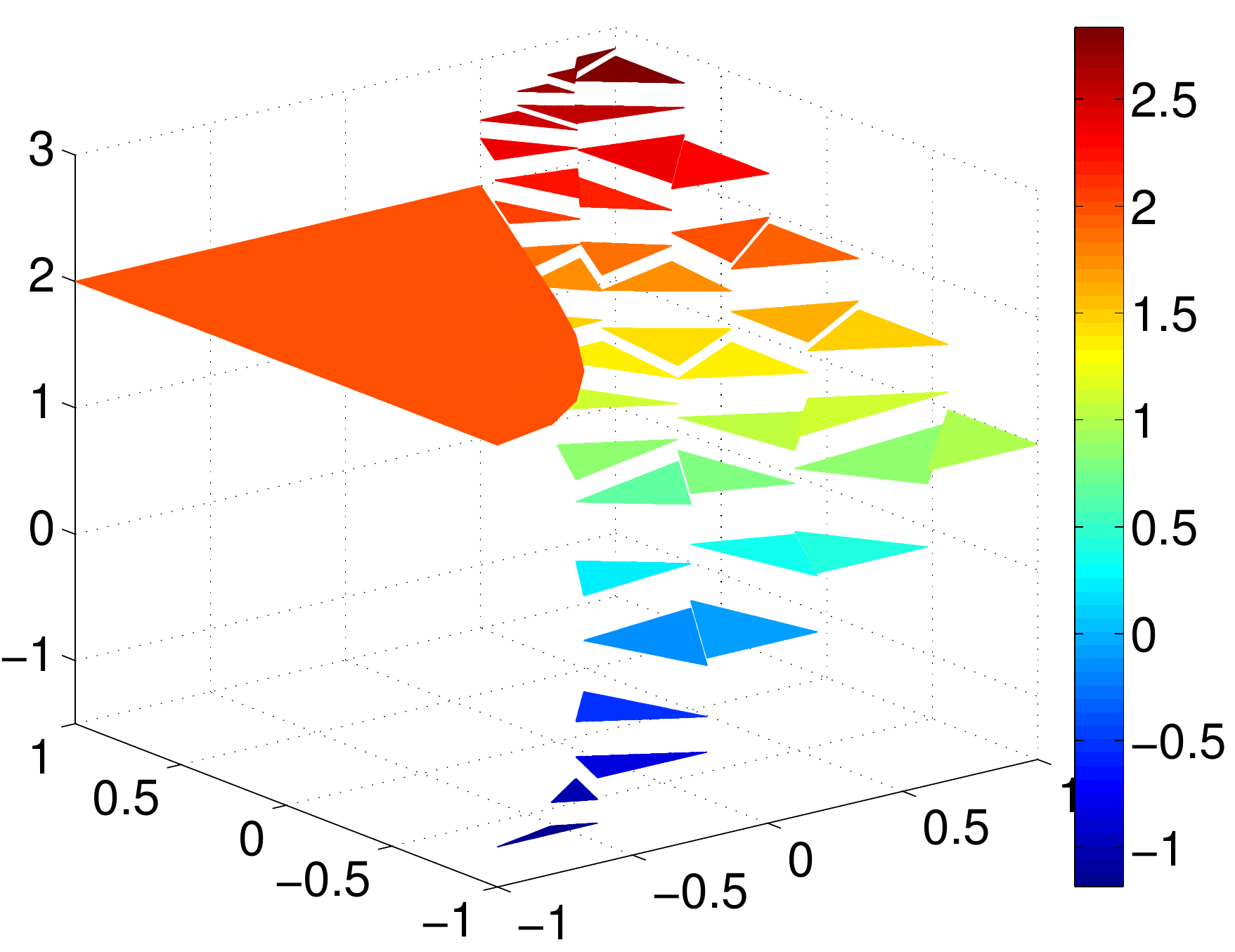}} \quad
  \resizebox{2.45in}{2.1in}{\includegraphics{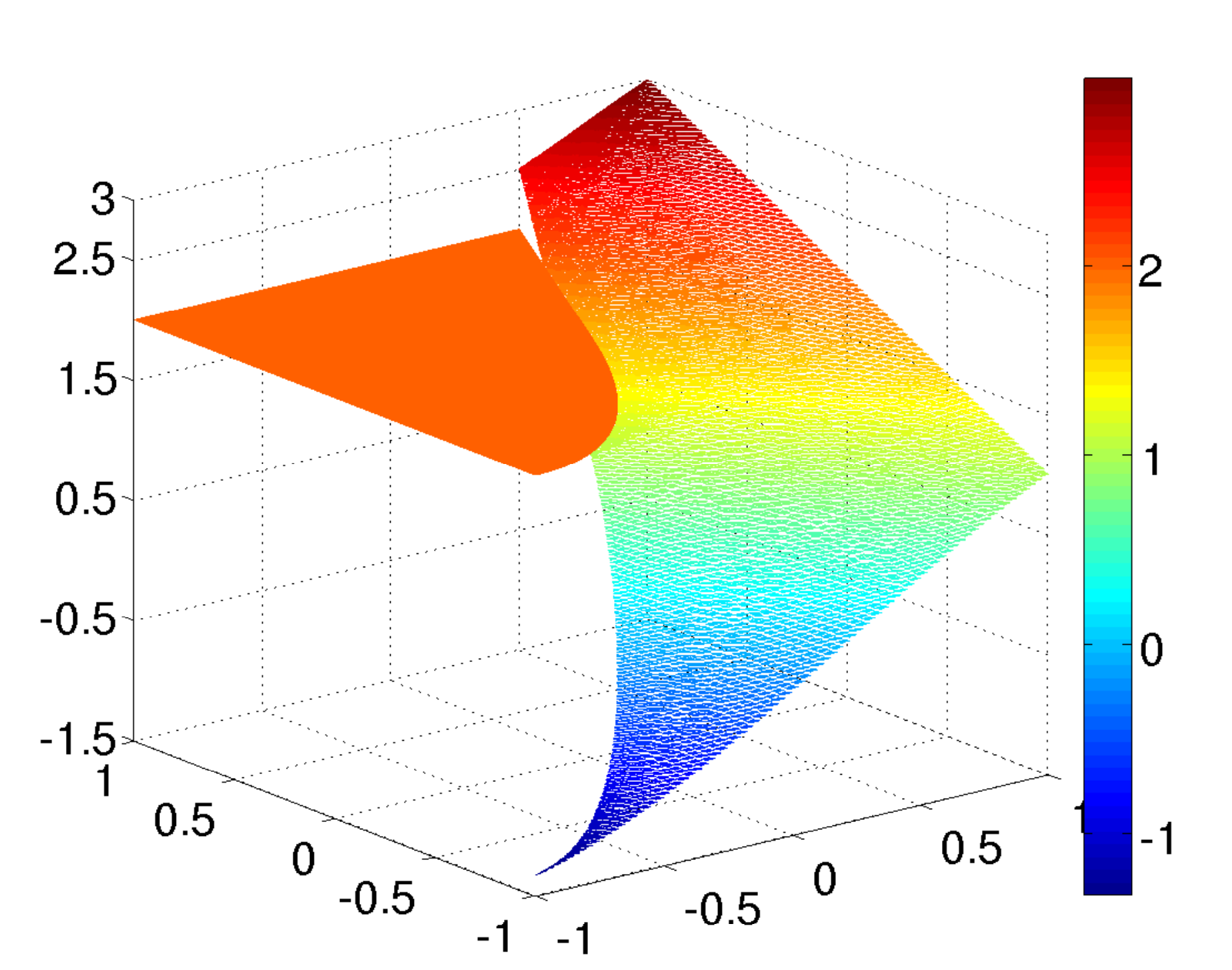}}
\end{tabular}
\caption{WG solutions of Example 6.
Left: Mesh level 1; Right: mesh level 5.}
\label{fig.ex6}
\end{figure}

\medskip

{\bf Example 6.} Here we consider the same interface and domain
geometry as in Example 5. The coefficient function $A(x,y)$ is
defined in the same way as in Example 5. We also fix the solution as
$u=2$ in $\Omega_1$. But in $\Omega_2$, the solution is given
differently by (see \cite{Hou:2010} for details)
$$
v(x,y)= \begin{cases}
\sin(x+y)+\cos(x+y),  \quad \mbox{ if }x+y\le 0\\
x+y+1,     \quad \quad \quad \quad \quad  \quad \quad   \mbox{ if }x+y>0.
\end{cases}
$$
Other necessary conditions can be derived similarly. Observe that
the solution $v=v(x,y)$ is $C^1$ continuous but not $C^2$ across the
line $x+y=1$.

The WG algorithm is formulated as in the previous example. In
particular, since $v(x,y)$ is $C^1$ across $x+y=1$, the function and
flux jumps are trivially zero across $x+y=1$. Thus, again, no
special boundary or interface treatment is carried out near the line
$x+y=1$. In fact, this line actually cuts through the finite element
triangles; see the left chart of Fig. \ref{fig.ex6}. It can be seen
that the loss of regularity in $v(x,y)$ is visually
indistinguishable in the WG solution on mesh level 5 (see the right
chart of Fig. \ref{fig.ex6}). In literature, a certain reduction of
convergence order has been reported for this $C^1$ continuous
example \cite{Hou:2010} when the standard Galerkin finite element
method is employed for a numerical approximation. In the present
study, it is found that the order of convergence remains unchanged
for the WG method, see Table \ref{tab.ex6}. This indicates that the
WG method preforms robustly for the challenging problem with low
regularity.

\begin{table}[!hb]
\caption{Numerical convergence test for Example 7.}
\label{tab.ex7}
\begin{center}
\begin{tabular}{||c||c|cc|cc||}
\hline\hline
Mesh & $\max \{h\}$ & \multicolumn{2}{c}{Solution}  & \multicolumn{2}{c||}{Gradient}   \\
\cline{3-4} \cline{5-6}
 & & $L_\infty$ error & order & $L_\infty$ error & order \\
\hline\hline
Level 1  & 6.5801e-01 & 3.4340e-02 &             & 3.2817e-01 &  \\ \hline
Level 2  & 3.4551e-01 & 1.1022e-02 & 1.7641 & 1.7896e-01 & 0.9413 \\ \hline
Level 3  & 1.7932e-01 & 3.5197e-03 & 1.7405 & 9.5039e-02 & 0.9650 \\ \hline
Level 4  & 9.1176e-02 & 1.0806e-03 & 1.7459 & 5.0095e-02 & 0.9468 \\ \hline
Level 5  & 4.6365e-02 & 3.3092e-04 & 1.7499 & 2.6003e-02 & 0.9696 \\
\hline\hline
\end{tabular}
\end{center}
\end{table}

\begin{figure}[!hb]
\centering
\begin{tabular}{cc}
  \resizebox{2.45in}{2.1in}{\includegraphics{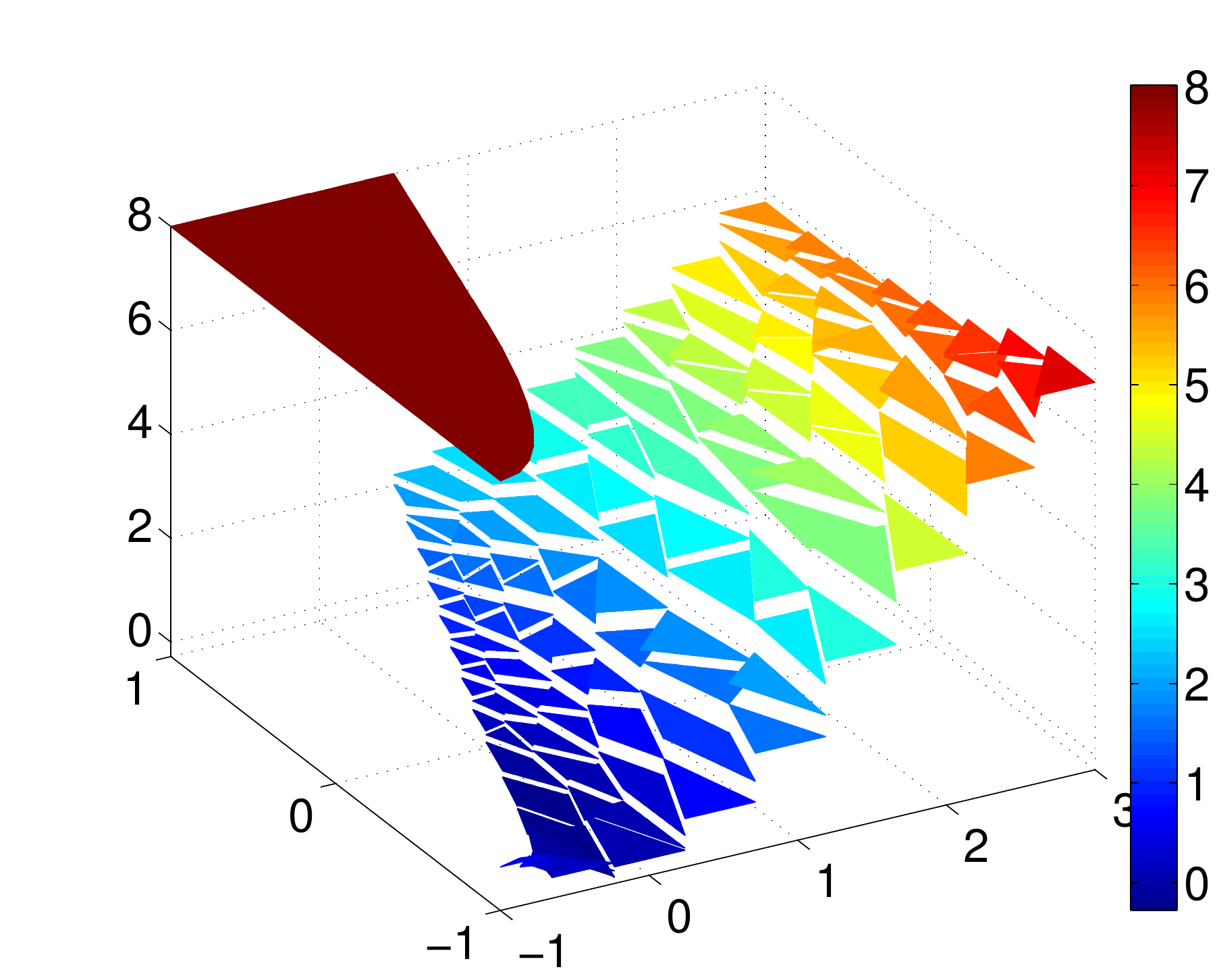}} \quad
  \resizebox{2.45in}{2.1in}{\includegraphics{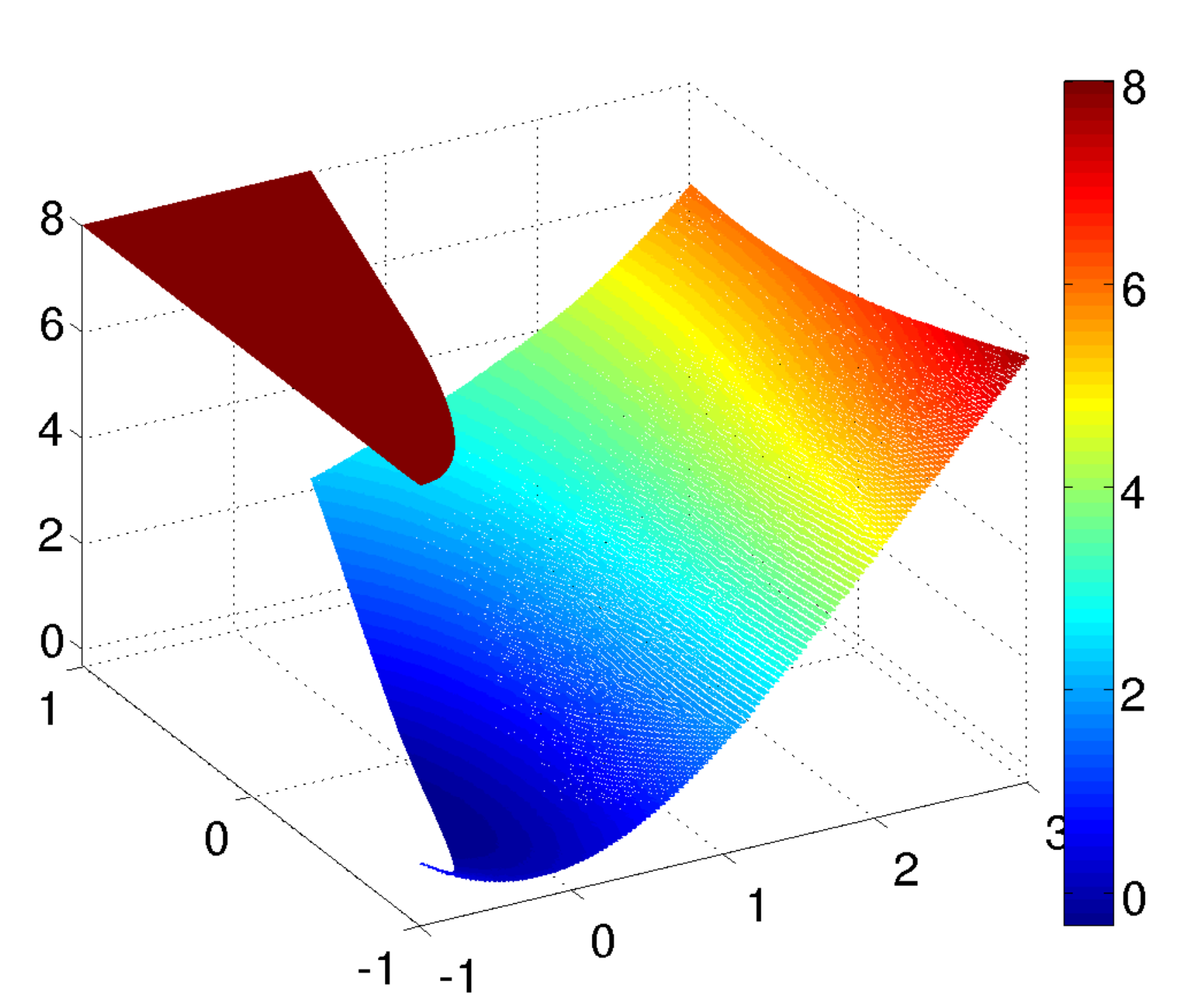}}
\end{tabular}
\caption{WG solutions of Example 7.
Left: Mesh level 1; Right: mesh level 5.}
\label{fig.ex7}
\end{figure}

\medskip

{\bf Example 7.} Consider again the same interface problem as in
Examples 5 and 6, over a larger domain $\Omega=[-1,3]\times[-1,1]$.
The coefficient function $A(x,y)$ is now defined to be $A_1(x,y)=1$
in $\Omega_1$ and $A_2(x,y)=2+\sin(x+y)$ in $\Omega_2$. The
analytical solution is give as \cite{Hou:2010}
\begin{align*}
u(x,y) & = 8 & \mbox{in } \Omega_1, \\
v(x,y) & = (x^2+y^2)^{5/6}+\sin(x+y) & \mbox{in } \Omega_2.
\end{align*}
Other necessary conditions can be derived from the analytical solutions.

The analytical solution is piecewise $H^2$ in this example. In
particular, the function $v(x,y)$ has a singularity at $(0,0)$ with
blow-up derivatives. Even though the analytical solution is well
behaved in other places, the singularity at the origin may cause
some order reduction in the finite element simulation
\cite{Hou:2010}. Due to this singularity, the convergence rate of
the WG method in the $L_\infty$ norm of the solution is reduced to
be about $1.75$ for the present example, see Table \ref{tab.ex7}. On
the other hand, the order of convergence for the WG method for the
gradient remains to be about one. The numerical solutions at mesh
levels 1 and 5 are depicted in Fig. \ref{fig.ex7}. Interestingly, in
both levels, the singularity at the origin is invisible. But the
numerical algorithm can sense this singularity via the convergence
rate.


\begin{table}[!hb]
\caption{Numerical convergence test for Example 8.}
\label{tab.ex8}
\begin{center}
\begin{tabular}{||c||c|cc|cc||}
\hline\hline
Mesh & $\max \{h\}$ & \multicolumn{2}{c}{Solution}  & \multicolumn{2}{c||}{Gradient}   \\
\cline{3-4} \cline{5-6}
 & & $L_\infty$ error & order & $L_\infty$ error & order \\
\hline\hline
Level 1  & 5.6919e-01 & 1.4984e-02 &        & 3.3001e-02 &  \\ \hline
Level 2  & 2.8459e-01 & 4.4307e-03 & 1.7578 & 1.7812e-02 & 0.8897 \\ \hline
Level 3  & 1.4230e-01 & 1.2336e-03 & 1.7687 & 9.8402e-03 & 0.8561  \\ \hline
Level 4  & 7.1149e-02 & 3.2639e-04 & 1.8634 & 5.1905e-03 & 0.9228 \\ \hline
Level 5  & 3.5574e-02 & 8.5110e-05 & 1.9231 & 2.6699e-03 & 0.9591 \\
\hline\hline
\end{tabular}
\end{center}
\end{table}

\begin{figure}[!hb]
\centering
\begin{tabular}{cc}
  \resizebox{2.45in}{2.1in}{\includegraphics{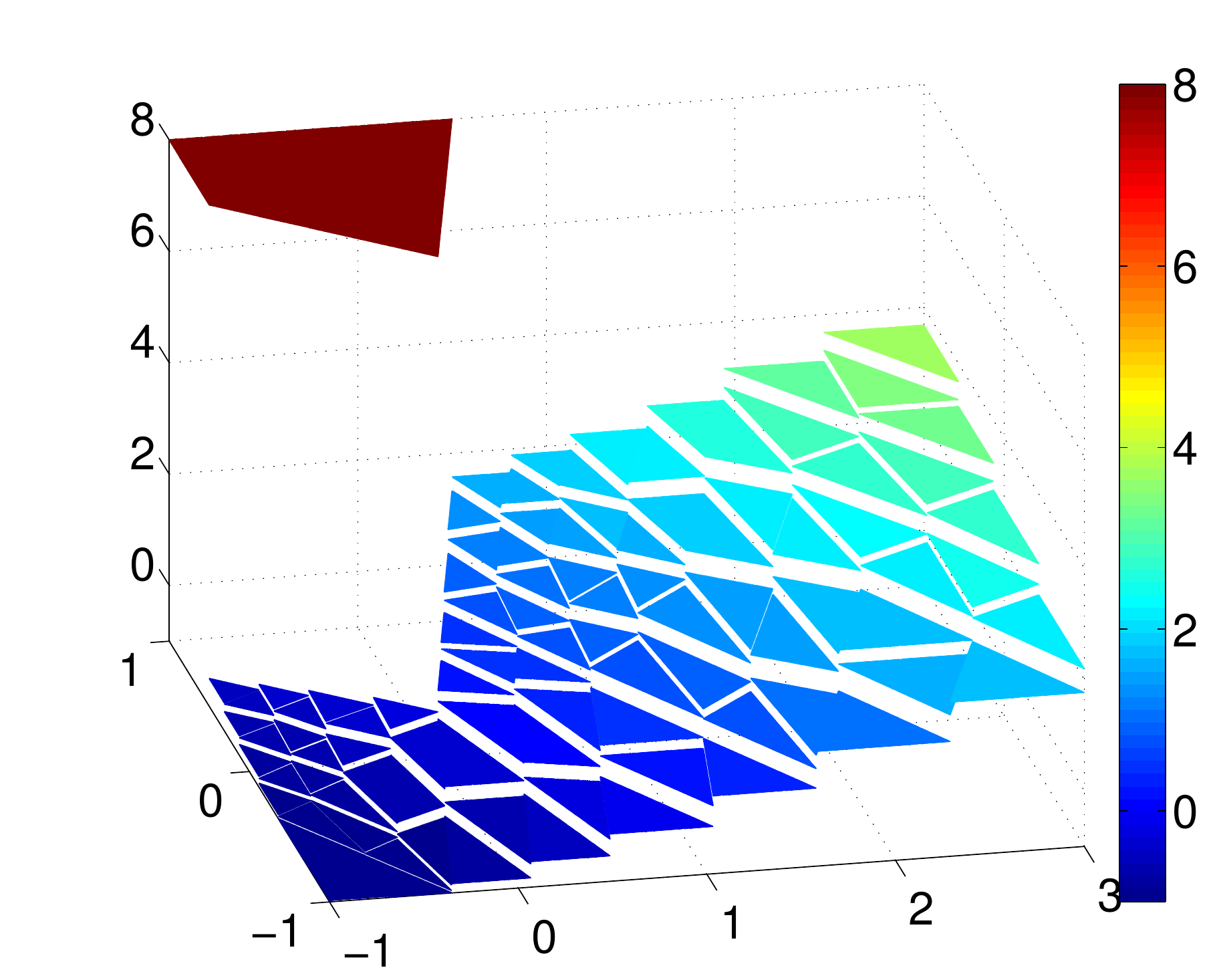}} \quad
  \resizebox{2.45in}{2.1in}{\includegraphics{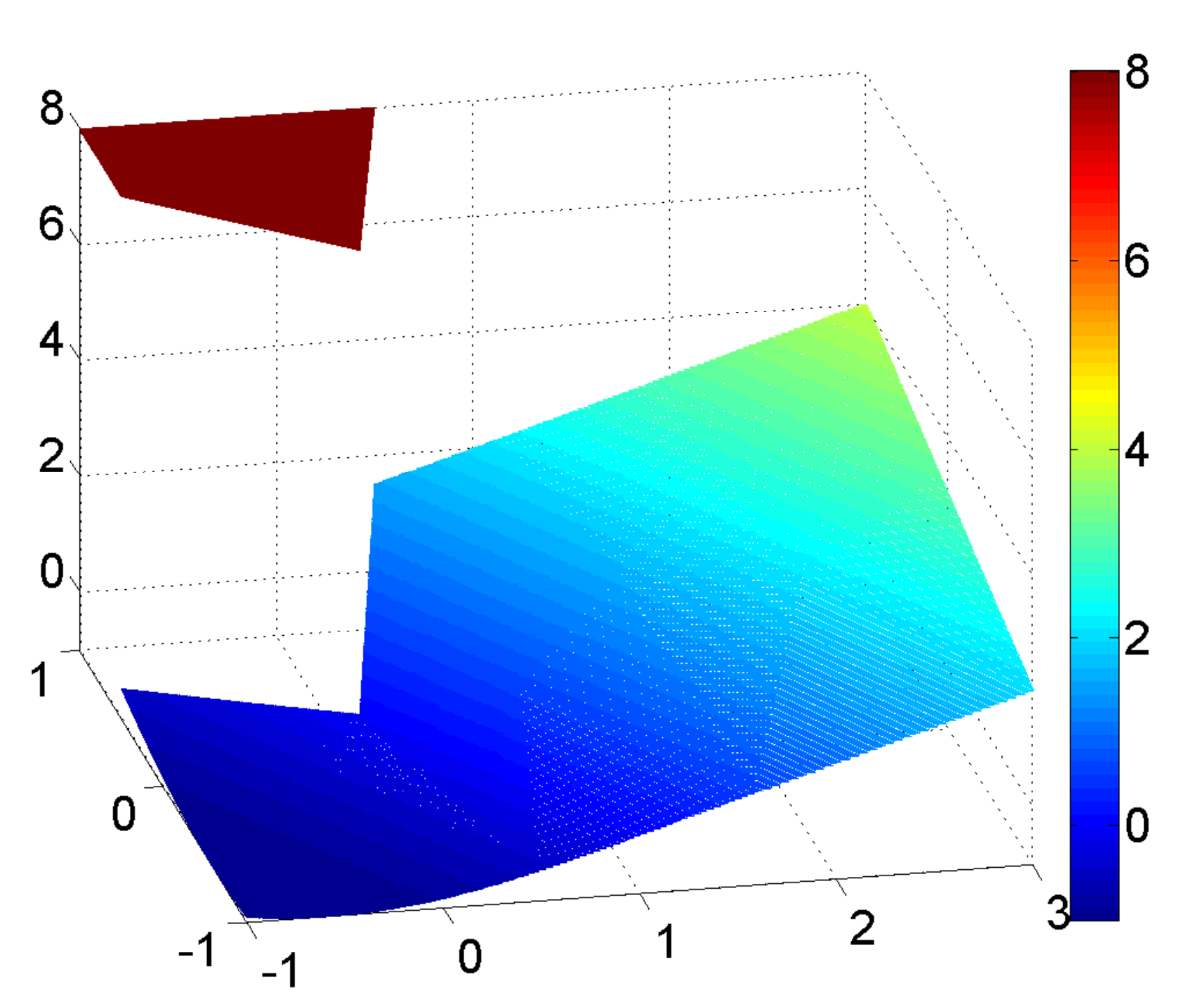}}
\end{tabular}
\caption{WG solutions of Example 8.
Left: Mesh level 1; Right: mesh level 5.}
\label{fig.ex8}
\end{figure}

\medskip

{\bf Example 8.} In this and next two examples, we continue to
examine the performance of the WG method for solutions of low
regularity when the interface is fixed to be a Lipschitz continuous
curve\cite{Hou:2010}. On a rectangular domain
$\Omega=[-1,3]\times[-1,1]$, we consider a case where the interface
$\Gamma$ consists of two pieces: $y=2x$ for $x+y>0$ and $y=-x/2$,
for $x+y\le 0$. Thus, the interface $\Gamma$ has a kink at $(0,0)$.
Denote by $\Omega_1$ the region on the left and upper part of
$\Gamma$, and $\Omega_2$ to be the rest of the domain. The
coefficient is given piecewisely by $A_1(x,y)=(xy+2)/5$ in
$\Omega_1$ and $A_2(x,y)=(x^2-y^2+3)/7$ in $\Omega_2$. The
analytical solution $u$ is fixed to be $u=8$ in $\Omega_1$, while in
$\Omega_2$, $v$ is given piecewisely by
$$ v(x,y)= \begin{cases}
\sin(x+y),  \quad \mbox{ if }x+y\le 0,\\
x+y,     \quad \quad \quad         \mbox{ if }x+y>0.
\end{cases}
$$
Other necessary conditions can be derived from the analytical solutions.

Similar to Example 5, the analytical solution is of $C^2$ continuous
but not $C^3$ in this example. The WG method yields uniformly second
and first order of accuracy in $L_\infty$ norm, respectively, for
the solution and its gradient as shown in Table \ref{tab.ex8}. The
kink of the interface $\Gamma$ is clearly seen in both charts of
Fig. \ref{fig.ex8}. In particular, on the coarsest mesh, the
alignment of finite element triangles to two line segments of
$\Gamma$ is obvious. No deterioration in solution occurs around the
interface corner.

\begin{table}[!hb]
\caption{Numerical convergence test for Example 9.}
\label{tab.ex9}
\begin{center}
\begin{tabular}{||c||c|cc|cc||}
\hline\hline
Mesh & $\max \{h\}$ & \multicolumn{2}{c}{Solution}  & \multicolumn{2}{c||}{Gradient}   \\
\cline{3-4} \cline{5-6}
 & & $L_\infty$ error & order & $L_\infty$ error & order \\
\hline\hline
Level 1 & 5.6919e-01 & 1.4326e-02 &        & 3.3514e-02 &   \\ \hline
Level 2 & 2.8459e-01 & 4.4250e-03 & 1.6949 & 2.0051e-02 & 0.7411 \\ \hline
Level 3 & 1.4230e-01 & 1.2350e-03 & 1.8411 & 1.0768e-02 & 0.8968 \\ \hline
Level 4 & 7.1149e-02 & 3.2693e-04 & 1.9175 & 5.5458e-03 & 0.9573 \\ \hline
Level 5 & 3.5574e-02 & 8.5118e-05 & 1.9415 & 2.8004e-03 & 0.9858 \\
\hline\hline
\end{tabular}
\end{center}
\end{table}

\begin{figure}[!hb]
\centering
\begin{tabular}{cc}
  \resizebox{2.45in}{2.1in}{\includegraphics{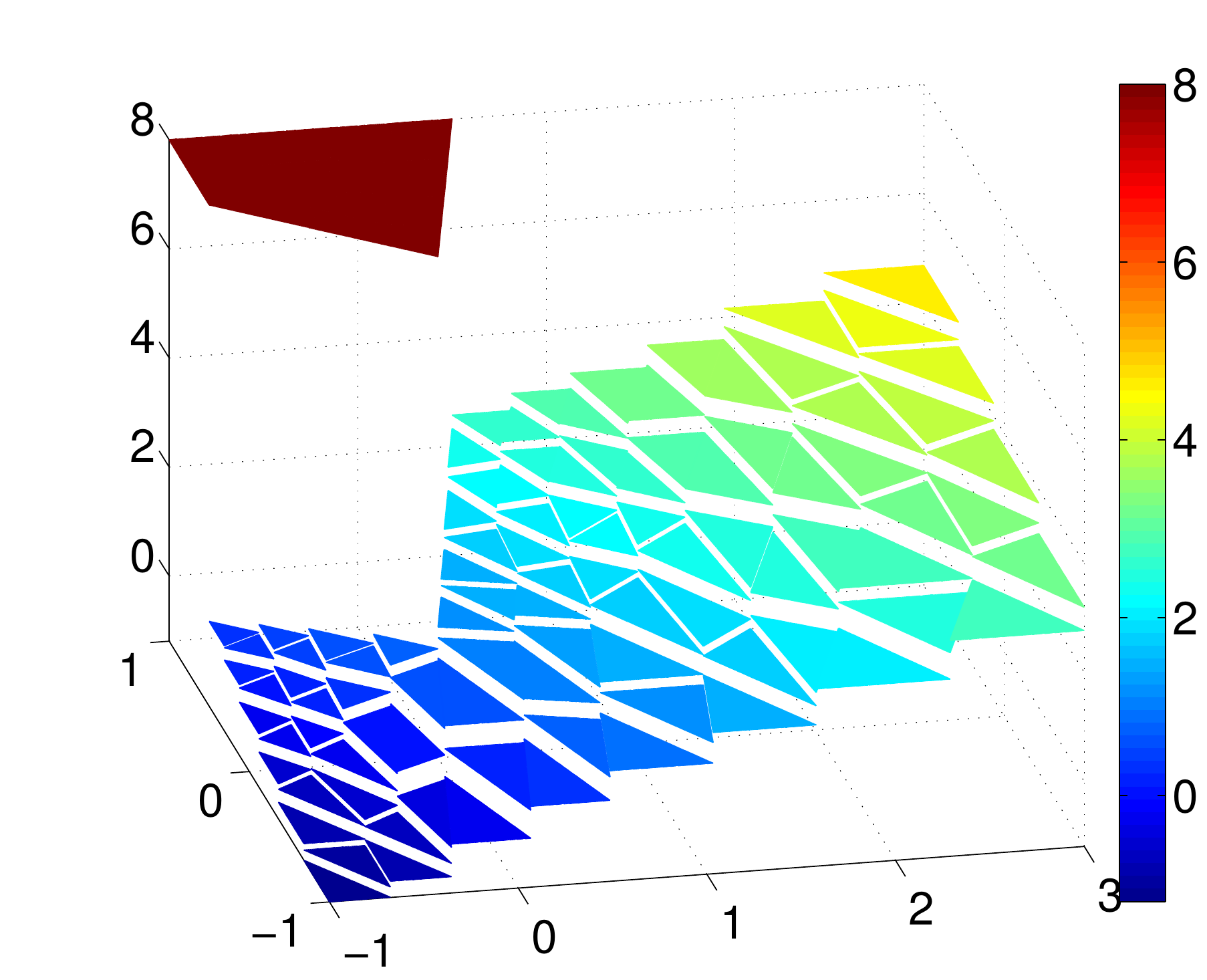}} \quad
  \resizebox{2.45in}{2.1in}{\includegraphics{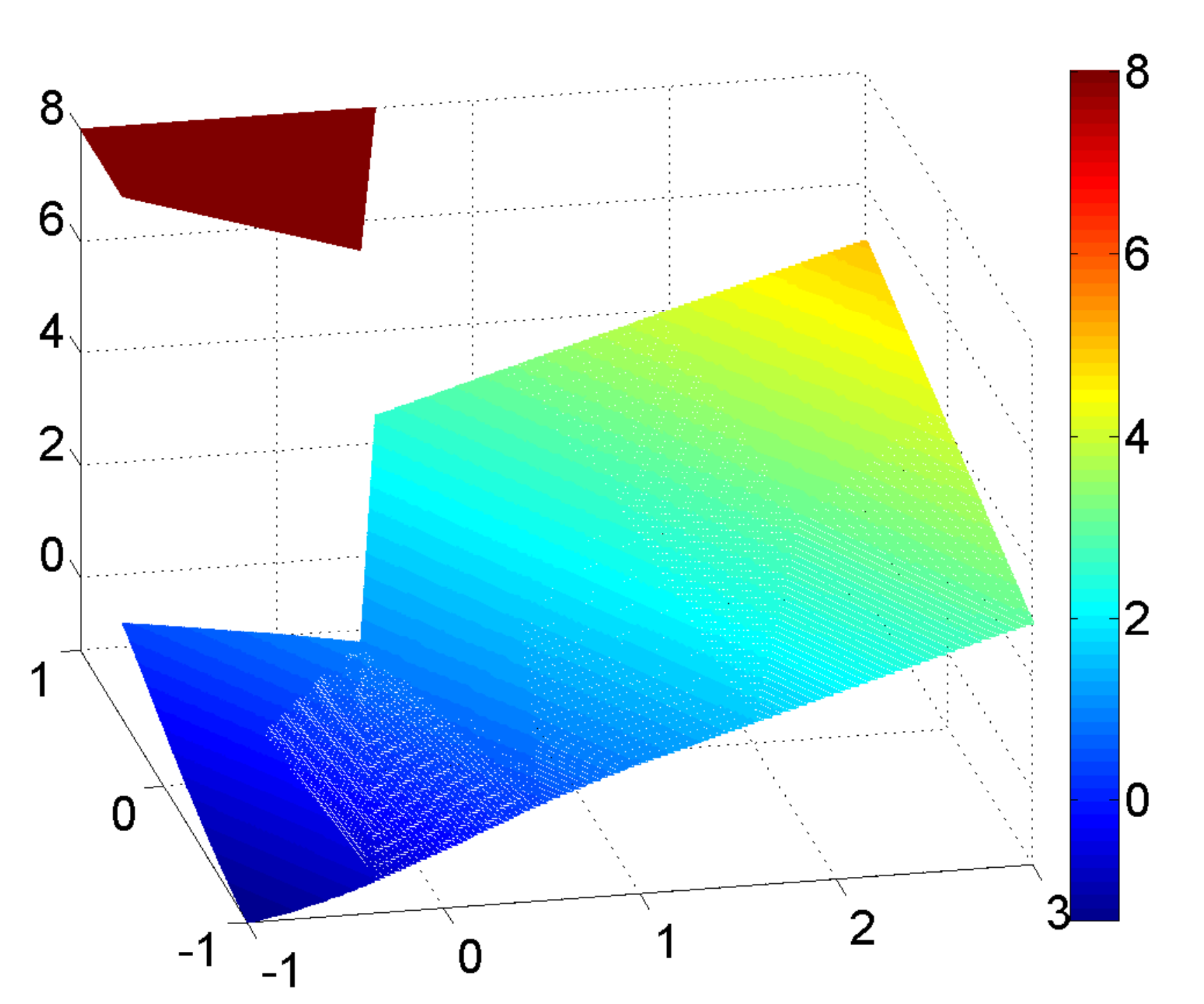}}
\end{tabular}
\caption{WG solutions of Example 9.
Left: Mesh level 1; Right: mesh level 5.}
\label{fig.ex9}
\end{figure}

\medskip

{\bf Example 9.} Consider an elliptic problem with the same
interface and domain geometry as in Example 8. The coefficient
function $A(x,y)$ is defined the same as in Example 8. Also, we fix
the solution to be $u=8$ in $\Omega_1$. But in $\Omega_2$, the
solution is given differently \cite{Hou:2010} by
$$
v(x,y)= \begin{cases}
\sin(x+y)+\cos(x+y),  \quad \mbox{ if }x+y\le 0,\\
x+y+1,     \quad \quad \quad \quad \quad  \quad \quad   \mbox{ if }x+y>0.
\end{cases}
$$
Other necessary conditions can be derived similarly. Like in Example
6, $v(x,y)$ is of $C^1$ continuous but not $C^2$ across the line
$x+y=1$. Similarly, no special numerical treatment is invoked near
the line $x+y=1$ so that this line actually cuts through finite
element triangles, see the left chart of Fig. \ref{fig.ex9}. It can
be seen from Table \ref{tab.ex9} that the WG method converges
uniformly with orders being close to second and first, respectively,
for the solution and its gradient in the $L_\infty$ norm.
Technically speaking, the interface is Lipschitz continuous in this
example, while it is $C^1$ continuous in Example 6. In other words,
the interface regularity is lower here than that in Example 6. However,
since body-fitted triangular meshes are employed in the WG method, the
interface with two simple line segments in the present example is
easier to be resolved by the body-fitted grids than the curved
interface in Example 6. The only issue that could make an impact on
the convergence is the geometrical singularity at the origin. The
last and present examples show that the WG method is very robust in
handling geometrical singularity, or sharp edged corners in general.

\begin{table}[!hb]
\caption{Numerical convergence test for Example 10.}
\label{tab.ex10}
\begin{center}
\begin{tabular}{||c||c|cc|cc||}
\hline\hline
Mesh & $\max \{h\}$ & \multicolumn{2}{c}{Solution}  & \multicolumn{2}{c||}{Gradient}   \\
\cline{3-4} \cline{5-6}
 & & $L_\infty$ error & order & $L_\infty$ error & order \\
\hline\hline
Level 1 & 7.1020e-01 & 4.0630e-02 &        & 7.1308e-02 &  \\ \hline
Level 2 & 3.5510e-01 & 1.1389e-02 & 1.8349 & 4.0849e-02 & 0.8038 \\ \hline
Level 3 & 1.7755e-01 & 3.3569e-03 & 1.7624 & 2.1170e-02 & 0.9483 \\ \hline
Level 4 & 8.8775e-02 & 1.0150e-03 & 1.7256 & 1.0780e-02 & 0.9737 \\ \hline
Level 5 & 4.4387e-02 & 3.1128e-04 & 1.7052 & 5.5740e-03 & 0.9516 \\
\hline\hline
\end{tabular}
\end{center}
\end{table}

\begin{figure}[!hb]
\centering
\begin{tabular}{cc}
  \resizebox{2.45in}{2.1in}{\includegraphics{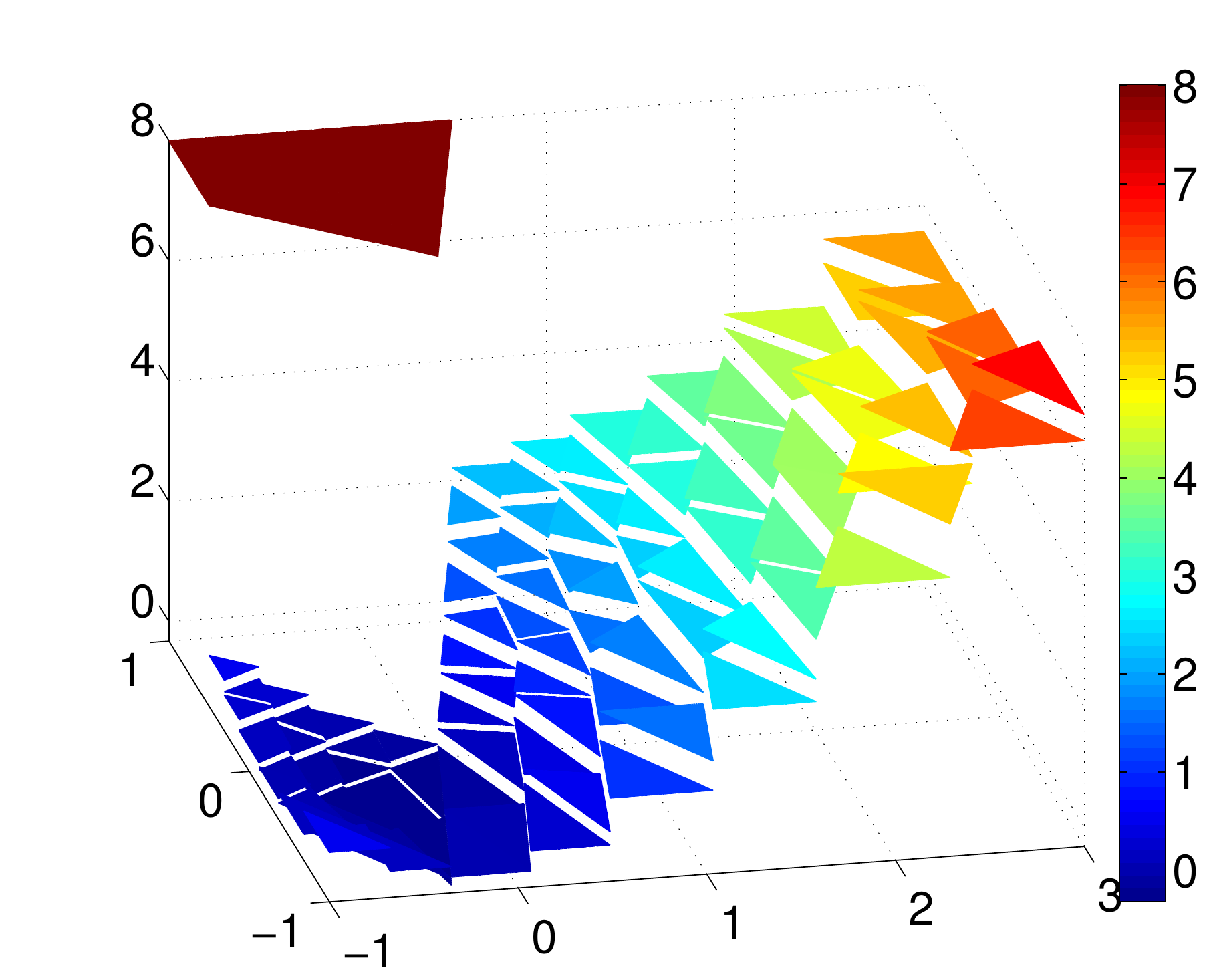}} \quad
  \resizebox{2.45in}{2.1in}{\includegraphics{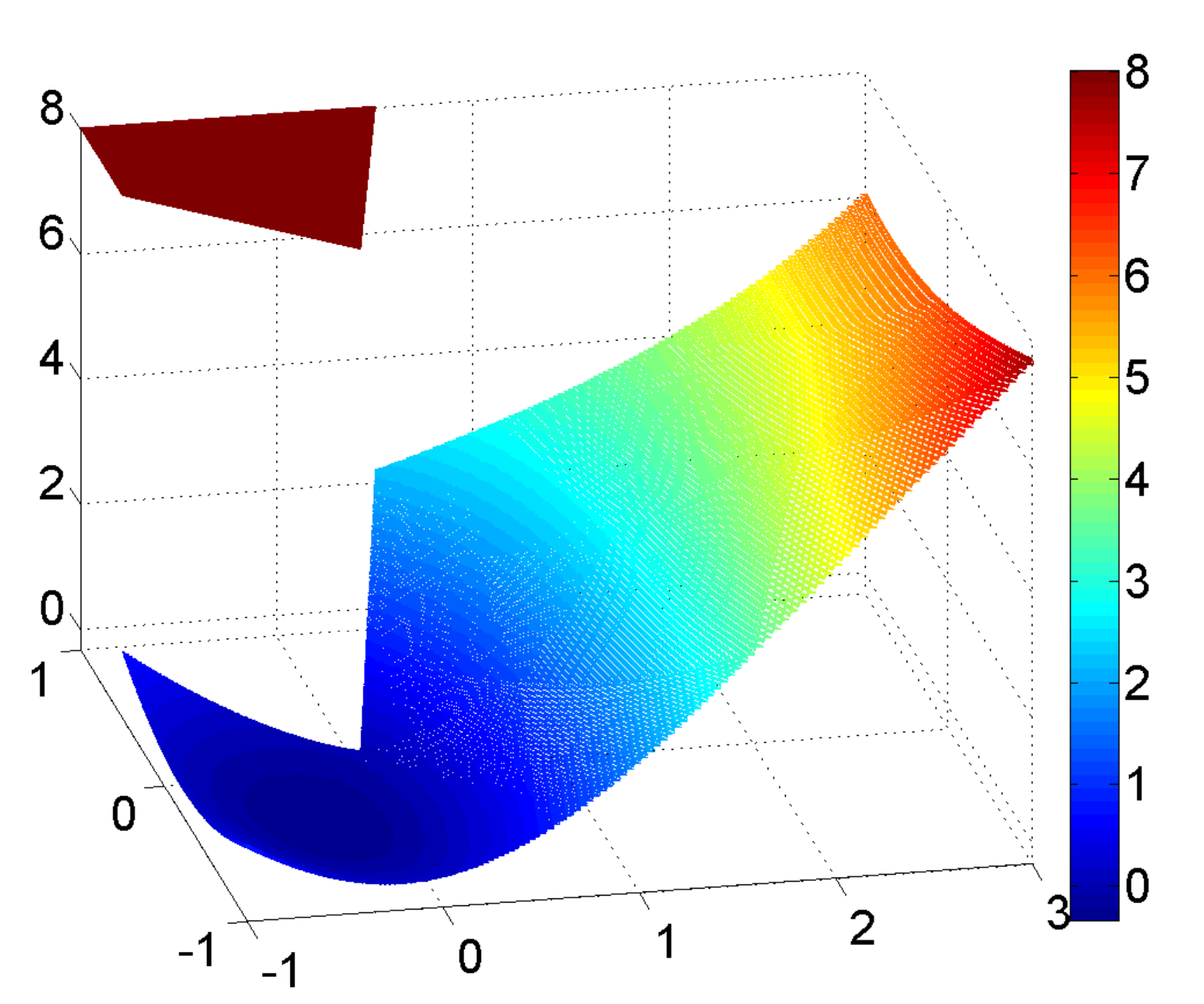}}
\end{tabular}
\caption{WG solutions of Example 10.
Left: Mesh level 1; Right: mesh level 5.}
\label{fig.ex10}
\end{figure}

\medskip

{\bf Example 10.} Consider the same interface and geometry as in
Examples 8 and 9. The coefficient function $A(x,y)$ is now defined
to be $A_1(x,y)=1$ in $\Omega_1$ and $A_2(x,y)=2+\sin(x+y)$ in
$\Omega_2$. The analytical solutions are give as \cite{Hou:2010}
\begin{align*}
u(x,y) & = 8 & \mbox{in } \Omega_1, \\
v(x,y) & = (x^2+y^2)^{5/6}+\sin(x+y) & \mbox{in } \Omega_2.
\end{align*}
Other necessary conditions can be derived from the analytical
solution.

Similar to Example 7, the analytical solution is piecewise $H^2$ in
the present case. The solution $v(x,y)$ has a singularity at $(0,0)$
with blow-up derivatives. The WG method achieves a similar order of
accuracy as that in Example 7. In particular, it can be seen in
Table \ref{tab.ex10} that the $L_\infty$ error in the solution is of
order $1.75$, while it remains to be about $1$ for the gradient
approximation. Both function singularity and geometrical singularity
are presented at the origin in this test; see the WG solutions in
Fig. \ref{fig.ex10}. Nevertheless, the WG method is capable of
delivering decent results at mesh level 1.

\begin{table}[!hb]
\caption{Summary of the overall orders of the WG method
for solutions with low regularities.}
\label{tab.sum}
\begin{center}
\begin{tabular}{||c||cc||cc||}
\hline\hline
Regularity & \multicolumn{2}{c||}{ $\Gamma$ is $C^1$ continuous}  &
\multicolumn{2}{c||}{$\Gamma$ is Lipschitz continuous}   \\
\cline{2-3} \cline{4-5}
& solution &  gradient  & solution & gradient \\
\hline\hline
$C^2$ & 1.8380 & 0.9056 & 1.8650  & 0.9069 \\ \hline
$C^1$ & 1.9523 & 0.9513 & 1.8487  & 0.8953  \\ \hline
$H^2$ & 1.7500 & 0.9558 & 1.7570  & 0.9193  \\
\hline\hline
\end{tabular}
\end{center}
\end{table}
\medskip

Finally, we would like to summarize the WG results from Example 5 to
Example 10 in Table \ref{tab.sum}. In each case, an overall
convergence rate is reported, which is calculated based only on mesh
level 1 and level 5. It can be seen from the $L_\infty$ error of the
gradient that the WG method essentially attains the first order of
accuracy in all test cases. For the $L_\infty$ error of the
solution, we analyze the order of the WG method case by case. When
the solution is $C^2$ continuous, the loss of regularity has no
impact on the order of convergence regardless whether the interface is
$C^1$ continuous or Lipschitz continuous. When the solution is $C^1$
continuous, the WG method achieves a very good order of accuracy for
$C^1$ continuous interface, while such order is slightly reduced
when the interface is Lipschitz continuous. However, the overall
order of convergence for the WG method for the latter case is still
very close to two. Thus, it is fair to say that the WG method
attains a second order of accuracy in $L_\infty$ norm for $C^1$
continuous solutions. On the other hand, the $L_\infty$ error of the
WG method is always around the order of $1.75$ when the solution is
$H^2$ continuous, for both $C^1$ and Lipschitz continuous
interfaces. In other words, even though the geometrical singularity
does not affect the convergence of the WG method, the function
singularity may do. Nevertheless, we note that a $1.75$th order for
solution with low regularities is still some of the best for this
class of challenging problems, to our best knowledge. Moreover,
although not reported in detail in this paper, the WG method achieves
the second order of accuracy in $L_2$ norm for the solution. Therefore,
the present study demonstrates a good robustness of the WG method for
solving elliptic interface problems. Readers are also encouraged to
draw their own conclusions from the results reported in this paper.

\section{Conclusion}\label{Sec:Conclusion}

The present paper presents some of the best results for solving
two-dimensional (2D) elliptic partial differential equations (PDE)
with low solution regularity resulted from modeling nonsmooth
interfaces. The weak Galerkin finite element method (WG-FEM) of Wang
and Ye \cite{JWang:2011} is introduced for this class of problems.
The WG-FEM is a new approach that employs discontinuous functions in
the finite element procedure to gain flexibility in enforcing
boundary and interface conditions. Such a strategy is akin to that
used by the discontinuous Galerkin (DG) methods. However, by
enforcing only weak continuity of variables though well defined
discrete differential operators, the WG-FEM is able to avoid pending
parameters commonly occurred in the DG-FEMs. Like traditional FEMs,
the WG-FEM employs body-fitted element meshes to represent
boundaries and material interfaces. Such an approach is convenient
for handling nonsmooth interfaces or geometric singularities. The
convergence analysis of the present WG-FEM for elliptic interface
problems is given.  The validity of the WG-FEM is further tested
over a large number of benchmark numerical tests with various
complex interface geometries and nonsmooth interfaces. The designed
second order accuracy is confirmed in our numerical experiments for
problems with smooth interfaces and sufficient solution
regularities.

Nonsmooth interfaces, such as sharp edges, cusps, and tips are
ubiquitous both in nature and in engineering devices and structures.
When nonsmooth interfaces are associated with elliptic PDEs, they
result in difficulties in theoretical analysis and in the design of
numerical algorithms. Unfortunately, to tackle the real-world
problems, dealing with nonsmooth interfaces is non-avoidable in
scientific computing. What aggravates the difficulty in error
analysis and in the numerical methods is the low regularity of
solution resulted from nonsmooth interfaces, which is a well-known
natural phenomena called tip-geometry effects in many fields. When
the interface is Lipschitz continuous and solution is $C^1$ or $H^2$
continuous, the best known result in the literature is of first
order convergence in solution and $0.7$ order of convergence in the
gradient \cite{Hou:2010}. It is demonstrated that the present WG-FEM
achieves at least an order of $1.75$ convergence in the solution and
an order of $1$ convergence in the gradient. A numerical study for
high order of WG methods should be conducted for a better
understanding of the method.

\end{document}